\documentclass[11pt,a4paper]{amsart}
\usepackage {amsmath, amssymb, a4wide, epsfig, enumerate, psfrag}
\usepackage[latin1]{inputenc}
\usepackage[english]{babel}
\usepackage[active]{srcltx}

\date{\today}

\author{Bertrand Deroin \and Romain Dujardin}
\thanks{B.D.'s research was partially supported by ANR-08-JCJC-0130-01,  ANR-09-BLAN-0116.}
\address{CNRS \\ D\'epartement de Math\'ematique d'Orsay \\ B\^atiment 425,
Universit\'e de Paris Sud,
91405 Orsay cedex, France.}
\email{bertrand.deroin@math.u-psud.fr}
\address{LAMA  \\
Universit\'e Paris-Est Marne-la-Vall\'ee  \\
5 boulevard Descartes \\
77454 Champs sur Marne  \\
France}
\email{romain.dujardin@univ-mlv.fr}

\title[Lyapunov exponents for surface groups representations]{Lyapunov exponents for surface groups representations.}

\newcommand{\cc}{\mathbb{C}}

\newcommand{\rr}{\mathbb{R}}
\newcommand{\dd}{\mathbb{D}}

\newcommand{\nn}{\mathbb{N}}
\newcommand{\pp}{\mathbb{P}}
\newcommand{\ee}{\mathbb{E}}
\newcommand{\hh}{\mathbb{H}}
 
\newcommand{\e}{\varepsilon}
\newcommand{\cv}{\rightarrow}

\newcommand{\fr}{\partial}
\newcommand{\om}{\Omega}
\newcommand{\set}[1]{\left\{#1\right\}}
\newcommand{\norm}[1]{\left\Vert#1\right\Vert}
\newcommand{\abs}[1]{\left\vert#1\right\vert}
\newcommand{\cd}{{\cc^2}}

\newcommand{\pu}{{\mathbb{P}^1}}
\newcommand{\rest}[1]{ \arrowvert_{#1}}

\newcommand{\unsur}[1]{\frac{1}{#1}}

\newcommand{\lrpar}[1]{\left(#1\right)}

\newcommand{\la}{\lambda}
\newcommand{\lo}{{\lambda_0}}

\newcommand{\La}{\Lambda}
\newcommand{\SL}{\mathrm{SL}(2,\mathbb C)}
\newcommand{\PSL}{\mathrm{PSL}(2,\mathbb C)}
\newcommand{\tbif}{T_{\mathrm{bif}}}
\newcommand{\cX}{\mathcal{X}}

\newcommand{\itm}{\item[-]}

\DeclareMathOperator{\supp}{Supp}
\DeclareMathOperator{\vol}{vol}

\DeclareMathOperator{\tr}{tr}

\DeclareMathOperator{\length}{length}

\DeclareMathOperator{\bif}{Bif}
\DeclareMathOperator{\stab}{Stab}
\DeclareMathOperator{\argcosh}{argcosh}

\newtheorem{prop} {Proposition} [section]
\newtheorem{thm}[prop] {Theorem}

\newtheorem{lem}[prop] {Lemma}
\newtheorem{cor}[prop]{Corollary}
\newtheorem{theo}{Theorem}

\newtheorem{propo}[theo]{Proposition}

\newtheorem{defprop}[prop]{Definition-Proposition}

\theoremstyle{remark}

\newtheorem{rmk}[prop]{Remark}

\begin{document}

\begin{abstract}
Let $(\rho_\la)_{\la\in \La}$ be a holomorphic family of representations of a surface group $\pi_1(S)$ into 
$\mathrm{PSL}(2, \mathbb{C})$, where $S$ is a topological (possibly punctured) surface with negative Euler characteristic. 
Given a  structure of Riemann surface of finite type on $S$ we construct a bifurcation current on the parameter space $\La$, 
that is a  (1,1) positive closed current attached to the bifurcations of the family. It is defined as the $dd^c$ of the Lyapunov exponent of the representation with respect to the Brownian motion on the Riemann surface 
$S$, endowed with its    Poincar\'e metric. We   show that this bifurcation current describes the asymptotic distribution of various
 codimension 1   phenomena in $\La$. For instance, the random hypersurfaces of $\La$ defined by the condition that 
   a random closed geodesic on $S$ is mapped under $\rho_\la$ to a parabolic element or the identity
    are asymptotically  equidistributed with respect to the bifurcation current.  The proofs are based on our previous work \cite{kleinbif}, and on a careful control of a discretization procedure of the Brownian motion. 
\end{abstract}

\maketitle

\section*{Introduction}

Let $G$ be a finitely generated group and $(\rho_\la)_{\la\in \La}$ be a holomorphic family of representations of $G$ 
into the M\"obius group $\PSL$.
Viewing these representations as dynamical systems on $\pu$, 
we can divide the parameter space $\La$  into a {\em stability locus}, where the action of $G$ is locally quasi-conformally structurally stable, and its complement, the {\em bifurcation locus}.

 In a previous article \cite{kleinbif}, under natural assumptions on the family $(\rho_\la)_{\la\in \La}$, 
we   showed how to design {\em bifurcation currents} on $\La$. These 
are positive closed currents aiming at an understanding of 
 the measurable and complex analytic structure of the bifurcation locus. 
 This construction relies on probabilistic 
ideas, let us briefly recall some of 
its key steps. We start by choosing a probability measure $\mu$ on $G$, which gives rise to a random walk  on $G$. An example is 
 the nearest neighbor random walk on the Cayley graph associated  to some symmetric  set of generators. Fix any matrix norm 
$\norm{\cdot}$ on $\PSL$.
Then the theory of random products of matrices implies that  for any non-elementary representation $\rho$, for almost every sample 
path $(g_n\cdots g_1)_{n\geq 1}$ 
of the random walk, $\unsur{n}\log \norm{\rho(g_n\cdots g_1)}$  converges to the {\em Lyapunov exponent} 
$\chi_\mu(\rho)$. It is easily seen
that  the Lyapunov exponent function $\la\mapsto \chi_\mu(\rho_\la)$ is plurisubharmonic (psh) on $\La$, 
and we define the  bifurcation current  by the formula $T_{{\rm bif},\mu} = dd^c_\la \chi_\mu(\rho_\la)$. 

It turns out that the support of $T_{{\rm bif},\mu}$ is 
precisely equal to the bifurcation locus. 
 Also, $T_{{\rm bif},\mu}$ describes the asymptotic 
 distribution of certain dynamically defined sequences of subvarieties in $\La$. Indeed,  
  for $t\in \cc$, and $g\in G$,   
let $Z(g,t)$ be the hypersurface in $\La$ defined by 
$$Z(g,t)  = \set{\la\in \La, \ \tr^2(\rho_\la(g)) = t}$$ ($\tr$ stands for  the trace). 
 The emblematic value is 
 $t=4$,   corresponding  to those parameters where $\rho_\la(g)$ becomes accidentally parabolic or equal to the identity, which 
are known to be dense in the bifurcation locus. Then the following  equidistribution statement holds:
for  almost every sample path   of the random walk, the sequence of divisors 
 $\unsur{2n}   Z( g_n\cdots g_1,t)$ converges to $T_{{\rm bif},\mu}$ when $n\cv\infty$. 

We note that analogous results were previously known in the context of holomorphic  families of 
rational mappings on $\pu$. An essential 
 motivation in \cite{kleinbif} was actually to fill the corresponding line in the famous Sullivan dictionary between 
rational iteration and Kleinian group theory. 
The reader is referred to \cite{survey} for an overview on these issues. 

\medskip

In this paper we specialize to the case where $G = \pi_1(S)$ is the 
fundamental group of an orientable surface of genus $g$ with $n$ punctures,   such that $\chi(S) = 2-2g-n<0$. 
As before we consider a  holomorphic family 
$(\rho_\la)_{\la\in \La}$ of non-elementary 
representations of $G$ into $\PSL$ (we can equally deal with  representations modulo conjugacy)
sending  simple loops about the punctures to parabolic transformations. 
Such representations will be referred to as {\em parabolic}. 
For the sake of  this introduction, let us  assume without further notice that these representations are faithful at 
generic parameters 
and not conjugate to each other in $\PSL$.
Our main purpose  is to construct a bifurcation current  on $\La$ 
{\em associated    to the choice of a conformal structure on} $S$. 
 For this, let  $X$ be a structure of Riemann 
surface of finite type on $S$ (that is, $X$ is biholomorphic to a compact Riemann surface minus $n$ points) and endow $X$ with its 
unique hyperbolic metric of curvature $-1$. Recall that it is of finite volume. This data in turn gives birth to a wealth of 
mathematical objects:
\begin{itemize}
\itm a dynamical system: the geodesic flow on the unit tangent bundle  $T^1X$;
\itm a stochastic process:   Brownian motion on $X$; 
\itm  an isomorphism between $G$ and a lattice $\Gamma$ 
in $\mathrm{PSL}(2, \rr)$ induced by the uniformization of $X$
\end{itemize}
We will demonstrate that our bifurcation current fits naturally with all these structures. Let us already point out that 
as in the previous case, this bifurcation current will be defined by taking the $dd^c$ of a certain Lyapunov exponent function.

\medskip

 A family of representations of $G$ (modulo conjugacy) naturally associated to the Riemann surface structure  is obtained  by 
 considering the 
 holonomies of holomorphic  projective structures on $X$. The parameter space for 
 this family of representations is biholomorphic to  $\cc^{3g-3+n}$  and is 
 the ambient space for the so-called Bers 
 embedding of Teichm\"uller space. As such it plays an  important role in Teichm\"uller theory (see Dumas 
 \cite{dumas} for a nice introduction to this topic).  Notice that the much studied 
    cases $(g=0, n=4)$ and $(g=1, n=1)$,  
corresponding to one-dimensional Teichm\"uller spaces, are associated to punctured surface groups. 
This is a motivation for allowing for non-compact surfaces, which, as we will see,   are a source of technical difficulty.

 We see that the space of projective structures 
  carries a natural bifurcation current (resp. Lyapunov exponent function). 
 The Sullivan dictionary predicts that the structure of this space should 
 resemble that of the space of polynomials in holomorphic dynamics. 
In a companion paper \cite{bers2}  we    introduce an  invariant, 
the {\em degree} of a projective structure, 
and show that it is closely related to the Lyapunov exponent of its holonomy. 
This   enables us to unveil a new analogy between spaces of projective structures and spaces of polynomials.

\medskip

Let us discuss  more precisely the contents of the present  paper. 
As before, the definition of the  bifurcation current depends on the construction of a Lyapunov exponent function on $\La$, 
that  we now describe. Mark a point $\star$ in $X$, 
 which will serve as a base point for the fundamental group, and equipp $X$ with its Poincar\'e metric. For  a Brownian path $\omega : [0,\infty) \rightarrow X$ starting at $\omega (0) = \star$, we do the following construction: for $t>0$, we append $\omega \rest{[0,t]}$ with a shortest path returning to $\star$, to  get a closed loop $\omega_t$, 
hence an element $[\omega_t]\in \pi_1(X,\star)$ --we disregard any possible ambiguity in the choice of the returning 
path.  


We then have the following result

\begin{propo}\label{propo:A}
Let $X$ be a hyperbolic Riemann surface of finite type 
and $\rho:G\cv\PSL$ be a non-elementary parabolic representation. Then there exists a positive real number 
$\chi_{\rm Brown}(\rho)$ such that for almost every Brownian path $\omega$ starting at $\star$, 
$ \unsur{t}\log \norm{\rho([\omega_t])}$ converges to $\chi_{\rm Brown}(\rho)$ when $t\cv\infty$. 
\end{propo}
 
In any holomorphic family $(\rho_\la)_{\la\in \La}$ of non-elementary representations of $G$, 
we thus obtain a Lyapunov exponent function 
$\la\mapsto \chi_{\rm Brown}(\rho_\la)$, and we define the bifurcation current associated to the Riemann surface structure $X$ 
as $\tbif = dd^c_\la(\chi_{\rm Brown}(\rho_\la))$. From now on we will refer to it as the {\em natural bifurcation current}.
We observe that this current depends only (up to a multiplicative constant) on the conformal structure on $X$, 
and not on the conformal metric on $X$ which was chosen to generate the Brownian motion 
(in the case where $X$ has punctures   some extra care is needed about the behaviour of the metric near the cusps). Indeed, by conformal invariance of the Brownian motion in dimension $2$, the Lyapunov exponent function associated to any other choice of conformal metric will be a constant multiple of $\chi_{\rm Brown}$ 
(see Remark \ref{rmk:clock}  for details). 


A similar procedure  yields a Lyapunov exponent $\chi_{\rm geodesic} =   \chi_{\rm Brown}$ 
 by using a generic geodesic starting at $\star$ instead of a Brownian path (with our normalization, the
 Brownian motion in the hyperbolic plane escapes at unit speed). In particular   $\tbif =   dd^c \chi_{\rm geodesic}$. A statement analogous to Proposition \ref{propo:A} in this context was already contained in the work of Bonatti, Gomez-Mont and Viana \cite{bonatti gomez mont viana} (at least when $X$ is compact).

Yet another way 
to define the Lyapunov exponent is to use the  isomorphism between $G$ and  a lattice 
$\Gamma\leq \mathrm{PSL}(2,\rr)$ provided by uniformization. Let $\hh$ be the hyperbolic plane.
For a sequence $r_n$ increasing  sufficiently fast   to infinity (e.g. $r_n = n^\alpha$ will do), 
we pick at random an element $\gamma_n\in \Gamma \cap B_\hh(0, r_n)$, relative to the counting measure. 
Then almost surely, $\unsur{2r_n} \log \norm{\rho(\gamma_n)}$ converges to $\chi_{\rm Brown}(\rho)$. 

\medskip

As in \cite{kleinbif}, the first main result asserts that $\tbif$ is indeed a bifurcation current. 

\begin{theo}\label{theo:support}
Let $X$ be a hyperbolic Riemann surface of finite type  and 
  $(\rho_\la)_{\la\in\La}$ be a holomorphic family of non-elementary parabolic 
representations of $G= \pi_1(X, \star)$. Then the 
support of the natural bifurcation current $\tbif$ is equal to the bifurcation locus. 
\end{theo}

The strategy of the proof is to rely on the corresponding support theorem in \cite{kleinbif}. For this, 
we use a discretization procedure for the Brownian motion, which was devised by Furstenberg \cite{furstenberg discretization} 
and then refined  by Lyons and Sullivan \cite{lyons sullivan}, among others. In some sense, it allows to project the Wiener measure 
(based, say, at $\star$), to a probability measure $\mu$ on $G$. We then 
show that the natural bifurcation current is equal to 
$T_{{\rm bif},\mu}$, up to a multiplicative constant, from which the result follows. 
The point is that the support theorem in \cite{kleinbif} requires an 
  exponential moment condition on $\mu$.  When  $X$ is compact, this fact is apparently part of the folklore.
We show that when  $X$ is merely of finite volume  and $\rho$ is a parabolic representation, $\mu$ still satisfies the following slightly
modified moment property
$$\int_G \exp\lrpar{ \beta d_\hh\big(0, \rho(g)(0)\big)} d\mu(g)<\infty \text{ for some } \beta>0,$$
 that is enough to apply the results of \cite{kleinbif} (see Proposition \ref{prop:moment}). 
 The proof is rather   technical, and occupies most of \S \ref{sec:brownian}. We believe that it could find applications beyond this work.

\medskip

The second main theme of the paper is the equidistribution of random hypersurfaces of the form $Z(g_n, t)$ 
towards the bifurcation current.  
Given the Riemann surface structure $X$, there are at least two natural choices for the group elements $g_n$:
\begin{itemize}
\itm A first possibility is to pick  randomly an element  $\gamma_n$ in $\Gamma\cap B_\hh(0, r_n)$ (according to the counting measure), and let $r_n\cv\infty$.
\itm A second choice is to choose at random a closed geodesic $\gamma_n$
of length at most $r_n$ on $X$. 
This geodesic defines a conjugacy class in $G$, so, abusing notation slightly, we can define  
  $\tr^2(\rho([\gamma]))$  and  it makes sense to consider  
 $Z(\gamma_n, t) =  \set{\la\in \La, \ \tr^2([\gamma_n]) = t}$. 
\end{itemize}
 
 In both cases, we show that the expected equidistribution holds. 
 For concreteness, let us state precisely the equidistribution theorem for random closed geodesics 
 (see Theorem \ref{thm:equidist} below for the other case). 
 The notion of a random closed geodesic will be made precise in  \S\ref{subs:models}. 
 Several natural models are possible,  the theorem applies to all of them.
 
 \begin{theo}\label{theo:equidist}
 Let $X$ be a hyperbolic Riemann surface of finite type  and 
  $(\rho_\la)_{\la\in\La}$ be a holomorphic family of non-elementary parabolic 
representations of $G= \pi_1(X, \star)$ into $\PSL$. Fix $t\in \cc$. Let $(r_n)$ be a sequence 
such that for every $c>0$, the series  $\sum_{n\geq 0} e^{-cr_n}$ converges. 
Let $(\gamma_n)$ be a sequence of 
 independent  random  closed geodesics on $X$ of length at most $r_n$. Then almost surely, 
$$\unsur{2 \length(\gamma_n)} [Z(\gamma_n, t)] \underset{n\cv\infty}{\longrightarrow }\tbif.$$
  \end{theo}

Notice  that this result does  {\em not} follow from 
 the equidistribution results in \cite{kleinbif} since the    random elements
 $\gamma_n$ are not obtained from a random walk on $G$. 
To prove the theorem,   we consider the sequence of psh potentials 
$$u(\gamma_n, \la) = \unsur{2\length(\gamma_n)} \log\abs{\tr^2(\rho_\la(\gamma_n))-t}.$$ We need to show that this sequence of psh functions a.s. converges  to $\chi_{\rm Brown}(\rho_\la)$ in $L^1_{\rm loc}(\La)$. 

The main step is to study this convergence   for a  {\em fixed} representation. 

\begin{theo}\label{theo:counting}
 Let $X$ be a hyperbolic Riemann surface of finite type 
and $\rho:G\cv\PSL$ be a non-elementary parabolic representation. Let $(r_n)$ and $(\gamma_n)$ be as in Theorem 
\ref{theo:equidist}. Then almost surely
$$\unsur{2\length(\gamma_n)} \log\abs{\tr^2(\rho (\gamma_n))}\underset{n\cv\infty}{\longrightarrow}
\chi_{\rm Brown}(\rho).$$
\end{theo}

Notice that when $\rho$ is the uniformizing representation (i.e. $\rho = \mathrm{id}$ when $G$ is identified to $\Gamma$), 
$ \log\abs{\tr^2([\gamma_n]))}$ is nothing but the  length of $ \gamma_n$ (up to an  error of order $O(e^{-\length(\gamma_n)})$). 
This theorem thus  describes how this quantity evolves as $\rho$ is deformed in the space  of representations. 

By the  Borel-Cantelli lemma,  to prove   Theorem \ref{theo:counting} we need an estimate on  the number of geodesics of length 
at most $r$ such that $\abs{\unsur{2\length(\gamma)} \log\abs{\tr^2(\rho (\gamma))} - \chi_{\rm Brown}(\rho)}
 \geq \e$. This counting is based on large 
deviation estimates for the traces of random products of matrices, together with  
a precise understanding of the Furstenberg-Lyons-Sullivan discretization procedure. 
Observe that the exponential moment property of the discretization measure is also crucial here.


Finally,  Theorem \ref{theo:equidist} is a rather straightforward consequence of
Theorem \ref{theo:counting}, by using some standard compactness properties of psh functions. 

\medskip

The outline of the paper is as follows. In \S\ref{sec:prel}  after  a few  facts on 
  hyperbolic geometry and $\PSL$ representations, we recall some necessary results 
  from \cite{kleinbif}. We also  discuss the basic properties of   
 Brownian motion on hyperbolic surfaces.  
In \S \ref{sec:brownian}, we define the natural Lyapunov exponent and prove Theorem \ref{theo:support}. 
For further reference as well as for completeness, we include
 a detailed treatment of the discretization procedure.  The equidistribution theorems are established in \S\ref{sec:equidist}. We conclude the paper with a  few open questions. 
 
 \medskip
 
 The results of this paper were announced in a preliminary form in \cite[\S 5.1]{kleinbif} and \cite[\S 4.5]{survey}.

\section{Preliminaries}\label{sec:prel}

Throughout the paper 
we use the following notation: if $u$ and $v$ are two  real valued functions, we write
 $u \asymp v$ (resp $u\lesssim v$)
  if there exists a constant $C>0$ such that $\unsur{C}u\leq v\leq Cu$ (resp. $u\leq Cv$). 
We also use the letter $C$ for a generic positive ``constant", which may change from line to line, independently of a varying parameter which should be clear from the context.

\subsection{Plane hyperbolic geometry} 
To fix notation, we summarize a few basic facts from hyperbolic geometry. See Beardon \cite{beardon} for a systematic account. 
We let $\hh^2$ (often simply $\hh$) be the hyperbolic plane (resp. $\hh^3$ be the hyperbolic 3-space) of curvature $-1$. The hyperbolic distance is denoted by $d_\hh$. The  hyperbolic area in $\hh^2$ is simply denoted by by $\mathrm{Area}(\cdot)$, and we recall that the area of a ball of radius $r$ is $4\pi \sinh^2\frac{r}{2}$.
As usual, we may  switch from the Poincar\'e disk (resp. ball) to the upper half plane (resp. space) model depending on the particular situation.  The  group of orientation preserving isometries of $\hh^2$ (resp. $\hh^3$)
will be viewed as $\mathrm{PSL(2, \rr)}$ (resp. $\PSL$). The $\mathrm{SO}(2)$ (resp. $SU(2)$) invariant point will be denoted by $0$.
 We let $\norm{\cdot}$   be the   
  operator norm associated to the standard Hermitian metric on $\cd$; recall that $\norm{A}$ is the spectral radius of $A^*A$. 
  The norm of an element   $\gamma\in \PSL$ is 
by definition $\norm{A}$, where $A$ is any matrix representative of $\gamma$ in $\SL$. 
 We also sometimes use $\norm{\cdot}_2$ defined by  $\norm{\gamma}_2:=  \big(\abs{a}^2+ \abs{b}^2+\abs{c}^2 + \abs{d}^2\big)^{1/2} $, where $A = \lrpar{\begin{smallmatrix}
   a & b \\ c& d                                                                                                                                                                                                                                                                                                                                                                                                                                                                                                                                                                                                                                                      \end{smallmatrix}}$. Likewise, if $\gamma\in \PSL$, $\tr^2(\gamma)$ and $\abs{\tr(\gamma)}$ are well-defined quantities. 
Norm and displacement are related by the identity
\begin{equation}\label{eq:norm distance}
\norm{\gamma}_2^2 = 2 \cosh (d_\hh(0, \gamma 0)), 
  \end{equation} 
in particular  $2 \log \norm{\gamma} \sim    d_\hh(0, \gamma 0)$ as $\norm{\gamma}\cv\infty$.
If $\gamma$ is a hyperbolic (resp. loxodromic) element then we denote its translation length by  $\ell(\gamma)$, which satisfies 
\begin{equation}\label{eq:translation length}
\cosh \lrpar{\frac{\ell(\gamma)}{2}} = \unsur{2} \abs{\tr(\gamma)}.
\end{equation}

\subsection{Spaces  of representations of ${\pi_1(S)}$ into $\PSL$}
A good source for the material in this section is \cite{kapovich}. 
Let $S$ be a (topological) surface of finite topological type, that is, homeomorphic to a compact 
surface of genus $g$ with    a finite number $n$ of  punctures.  Throughout the paper we
 always assume that $\chi(S) = 2-2g-n<0$. If $S$ is endowed with a complex structure, the corresponding Riemann surface will be denoted by $X$. Fix a basepoint $\star\in S$ 
  and put 
 $G =\pi_1(S, \star)$. The space of representations of $\mathrm{Hom}(G, \PSL)$ 
 is an affine complex algebraic variety (recall that $\PSL$ is affine because $\PSL\simeq \mathrm{SO(3, \cc)}$). A general holomorphic family of representations is simply a holomorphic mapping $\La\cv \mathrm{Hom}(G, \PSL)$. 
 The   group $\PSL$ acts algebraically 
  by conjugation on $\mathrm{Hom}(G, \PSL)$,  and the algebraic quotient 
  $\mathrm{Hom}(G, \PSL)/\! /\PSL$ is the  {\em character variety} $\cX(G)$. 
  Coordinates on $\cX(G)$ are given by functions of the form $\rho\mapsto\tr^2(\rho(\gamma))$. 
We also let $\cX^{NE}(G)\subset \cX(G)$ be the subset corresponding to non-elementary representations.

When $S$ is not compact we actually   rather work with the {\em relative character variety} $\cX_{\rm par}(G)$, which is defined as follows. Let  $(h_i)_{i =1\ldots n}$ in $G$ be  classes of closed loops turning once around the punctures. We define the set of {\em parabolic} 
representations    by 
 $$\mathrm{Hom}_{\rm par}(G, \PSL) = \set{\rho\in \mathrm{Hom}(G, \PSL), \ \forall i, \ 
 \rho(h_i) \text{ is parabolic}}.$$   The algebraic quotient 
 $\cX_{\rm par}(G) =  \mathrm{Hom}_{\rm par}(G, \PSL)/\! /\PSL$ is the  relative character variety. 
 It can be shown that $\cX_{\rm par}(G)$ is  of complex dimension  $6g-6+2n$, and is smooth at non-elementary representations.  

\subsection{Bifurcation currents}
Here we gather some material from \cite{kleinbif}. See also \cite{survey} for a general account on   bifurcation currents in one-dimensional dynamics, and \cite{beardon, kapovich} for basics on Kleinian (and more generally M\"obius) groups.

Let $G$ be as above, $\La$ be a connected complex manifold and $\rho = (\rho_\la)_{\la\in \La}$ be a holomorphic family of representations of $G$ into $\PSL$. We usually work under the following assumptions:
\begin{enumerate}
\item[(R1)] the family is non-trivial, in the sense that there exists $\la_1$, $\la_2$ such that 
the representations  $\rho_{\la_i}$, $i=1,2$ are not   conjugate  in $\PSL$;
\item[(R2)] there exists  $\la_0\in \La$ such that  $\rho_{\la_0}$ is faithful;
\item[(R3)] for every $\la\in \La$, $\rho_\la$ is non-elementary. 
\end{enumerate}
It may sometimes be useful to   replace (R3) by the following weaker statement
\begin{enumerate}
\item[(R3')]  there exists  $\la_0\in \La$ such that  $\rho_{\la_0}$ is non-elementary.
\end{enumerate}
In that case the subset $E\subset\La$ of elementary representations is a proper real-analytic subset. 
According to Sullivan \cite{sullivan acta} there exists  a partition $\La = \bif\cup\stab$ of $\La$ into an open   locus  $\stab$ of \textit{stable representations} (in the sense of dynamical systems)  and a closed \textit{bifurcation locus} $\bif$.  By definition, a representation $\rho$ is stable  if nearby representations in $\La$ are quasi-conformally conjugate to $\rho$. 

\medskip

Fix a probability measure $\mu$ on $G$, whose support  generates $G$ as a semi-group. We will assume  the following locally uniform exponential moment condition:
\begin{equation}\tag{M}
 \forall\lo\in \La, \
 \exists N(\lo)  \text{ neighbd of }\lo,\text{ and }\sigma >0 \text{ s.t. } \forall \la\in N(\lo), 
\int_G\norm{\rho_\la(g)}^\sigma  d\mu(g) <\infty,
\end{equation} where $\norm{\cdot}$ is any matrix norm on $\PSL$.
Under such a moment condition one can define a Lyapunov exponent  function on $\La$ by the formula
$$\chi_{\mu}(\la) = \lim_{n\cv\infty} \unsur{n} \int \log\norm{\rho_\la(g)} d\mu^n(g),$$ and the associated bifurcation current $ T_{{\rm bif}, \mu} = dd^c\chi_\mu$ --we deliberately  emphasize the dependence on $\mu$ here.

The following result was established in \cite[Thm. 3.13]{kleinbif}.

\begin{thm}\label{thm:support}
Let $(G, \rho, \mu)$ be a holomorphic family of representations of $G$, satisfying (R1, R2, R3'), endowed with a measure $\mu$, generating $G$ as a semi-group, and satisfying (M). Then $\supp(T_{{\rm bif}, \mu} )$ is equal to the bifurcation locus. 
\end{thm}

A comment is in order here: the theorem was stated in \cite{kleinbif} under the more restrictive moment condition 
\begin{equation}\label{eq:exponential moment in the group}
 \text{there exists } s>0 \text{ s.t. }\int_G \exp(s\;  \length (g)) d\mu(g) <\infty,
\end{equation}
 where the length is relative to some finite  system of generators of $G$. But, as explained in the beginning of \cite[\S 3.3]{kleinbif}, it still holds under (M). 
 
 \medskip
 
 It is important to observe that these notions also make sense on (resp. relative) character varieties. Indeed let $
 \Lambda$ be a piece of complex submanifold of the set of characters of non elementary 
 representations of $G$ into $\PSL$, $\Lambda\subset \mathcal{X}^{NE}(G)$ 
 (resp. $\Lambda\subset \mathcal{X}_
 {\rm par}^{NE}(G)$). Then it can be lifted to a holomorphic family of representations, transverse to 
 the foliation by conjugacy classes, so we can define a stability (resp. bifurcation) locus on $\La$, which does not depend on the choice of lifting. 
 Likewise, the validity of the moment condition 
 (M), as well as the value of the Lyapunov exponent, are independent of the lift. Hence there is a well-defined bifurcation current on $\La$, and Theorem \ref{thm:support} holds. 
  
 \subsection{Brownian motion}   \label{subs:brownian} 
   More details on the material in this paragraph  can be found in \cite{candel conlon, grigoryan, franchi le jan} 
 Let $X$ be a complete hyperbolic Riemann surface, endowed with its hyperbolic metric. Brownian motion on $X$ is defined as the Markov 
 process with transition probability $P_t(x, \cdot) =  p_t(x,y)dy$, where the volume element $dy$ is relative to the hyperbolic metric, and the
 heat kernel $p_t$ satisfies 
 $$\frac{\fr}{\fr t} p_t(x, \cdot) =   \Delta p_t(x, \cdot) \text{ and } \lim_{t\cv 0} p_t(x,\cdot)dx =\delta_x. $$
 (It is also common to work with $\frac12 \Delta$. We choose this normalization to stay in accordance with the usual convention in foliation theory (see \cite{candel conlon})).
 This process is a diffusion with continuous sample paths so    it gives rise for  every $x\in X$ to a probability measure $W_x$
 (the Wiener measure) on the space    $\om_x$ of continuous paths 
 $\omega:[0,\infty[\cv X$  such that  $\omega(0) =x$ (itself endowed with the compact open topology).
We also denote by $\om$ the space of  continuous paths $\omega:[0,\infty[\cv X$ and put $W = \int W_x dx$.
 We often the classical notation $\pp_x$ 
(resp.  $\ee_x$) for the probability (resp. expectation) of some event relative to $W_x$. 

It is   useful to think of sample paths of the Brownian motion on $X\simeq
\Gamma  \backslash \mathbb H^2$ as projections of Brownian paths on $\mathbb H^2$. To say it differently, 
 the Wiener measures associated to a point in $X$ and  of one of its lifts to $\hh^2$ are naturally identified. 
With our normalization, the linear drift of Brownian motion on $\hh^2$ is equal to 1, that is, 
for $W_0$--a.e. Brownian path $\omega$, $d_
\hh(0, \omega(t))\sim t$ as $t\cv\infty$ (see the appendix in \cite{deroin dupont} for an elementary proof).   

We will   require the following precise estimate for the radial density 
$k^0(t,r)$ 
 of the law at time $t$ of Brownian motion issued from $0$  (i.e. $P_t(0, \cdot)$) in  $\hh^2$ 
\begin{equation}\label{eq:davies}
k^0(t,r) \asymp (1+r^{-1})^{-1} t^{-1} (1+r+ t)^{-1} (1+r) \exp \lrpar{- \frac{(r-t)^2}{4t}}
\end{equation}
(see e.g.
 \cite[Thm 5.7.2]{davies spectral theory}).

 \section{The Lyapunov exponent associated to Brownian motion}\label{sec:brownian}

Let $S$ be a topological surface of genus $g$ with $n$ punctures, 
with a marked point $\star$,  and such that  $\chi(S)<0$.  Recall that $G= \pi_1(S, \star)$.
In  \S \ref{subs:def lyap} we show that the data of a structure of 
 Riemann surface of finite type on $S$ gives  rise to a natural 
Lyapunov exponent function $\chi_{\rm Brown}$ on $\mathrm{Hom}_{\rm par}(G, \PSL)$ (resp. $\cX_{\rm par}(G)$). 
Proving that $dd^c \chi_{\rm Brown}$ is a genuine bifurcation current occupies the remainder of this section. Our main tool is a 
 discretization procedure for Brownian motion, which will be discussed  thoroughly in  \S\ref{subs:discretization}. From this point, the main 
 issue is to obtain estimates for a probabilistic stopping time, associated to the discretization. The case where $X$ is compact is 
 substantially simpler, and will be discussed in \S\ref{subs:compact}. In \S\ref{subs:proof}, we deal with the general case. 
 The proof of Theorem \ref{thm:support brownien} is given in \S \ref{subs:proof proof}.
 In    \S\ref{subs:geodesic}, we re-interpret these results  from the geodesic point of view, from which Proposition \ref{propo:A} follows.

\subsection{Definition}\label{subs:def lyap}
Let $X$ be a  structure of Riemann surface of finite type on $S$ (i.e. a neighborhood of each puncture 
is biholomorphic  to $\dd^*$). The Euler characteristic assumption implies that $X$ is hyperbolic, and we endow $X$ with its 
hyperbolic metric of curvature $-1$. 
Let $(\widetilde X, \star)$ be the universal cover   of $(X, \star)$.  We fix an isomorphism
 between  $\widetilde X$ and $\hh^2$ mapping $\star$ to 0. Then the deck transformation group 
 $G$ becomes naturally isomorphic to a lattice $\Gamma\leq \mathrm{PSL(2,\rr)}$. 
 We let $\rho_{\rm can}$ be the (Fuchsian) uniformizing representation of $G$ 
such that $\Gamma  = \rho_{\rm can}(G)$, and in the sequel we conveniently identify $G$ and $\Gamma$ when needed. 
 
 
Fix a representation $\rho\in \mathrm{Hom}(G, \PSL)$. 
If $\gamma$ is any loop on $X$ based at $\star$, and $[\gamma]\in G$ denotes its class (from now on we drop the square brackets), we may consider the quantity $\log \norm{\rho(\gamma)}$. 
%
%
\medskip

We are now ready to define the natural  Lyapunov exponent of a parabolic representation  of $G$, associated to the complex structure on $S$. Justifying the validity of this definition will occupy the remainder of this subsection. 

 \begin{defprop}\label{def:lyap}
Let $X$ be a hyperbolic Riemann surface of finite type and $\rho:\pi_1(X, \star)\cv \PSL$ be a parabolic representation

Define a family of loops as follows: for $t>0$, consider a Brownian path $\omega$ issued from $\star$, and concatenate $\omega\rest{[0, t]}$ with a shortest geodesic joining $\omega(t)$ and $\star$, thus obtaining a closed loop $\widetilde \omega_t$. 

Then for $W_\star$ a.e. $\omega$ the limit 
\begin{equation}\label{eq:deflyap}
\chi_{\rm Brown}(\rho) = \lim_{t\cv\infty} \unsur{t}\log \norm{\rho\lrpar{\widetilde \omega_t}}
\end{equation}
 exists and does not depend on $\omega$.  
This number is by definition the Lyapunov exponent of $\rho$, associated to the complex structure 
$X$.  
 \end{defprop}

Notice that there is a little ambiguity in the definition for there might be several shortest paths joining 
  $\omega(t)$ and $\star$. As it will be clear from the proof, this does not affect the existence nor the value of the limit in \eqref{eq:deflyap}. It will also follow from the proof that $\chi_{\rm Brown}$ does not depend on $\star$. 
  Since with our normalization,  the escape rate of the  Brownian motion in $\hh^2$ is $1$, we deduce from 
 \eqref{eq:norm distance} that $\chi_{\rm Brown}(\rho_{\rm can}) = \frac12$. 
 
  \medskip
 
 To prove the existence of $\chi_{\rm Brown}$ and study its properties, we will need to consider the holonomy of non-closed paths on $X$. For this, we work on  the {\em suspension} of $\rho$.
  For any representation $\rho\in \mathrm{Hom}(G, \PSL)$, the {suspension} of $\rho$ 
is by definition 
the quotient of  $\mathbb H^2 \times \pu$ under  the diagonal action of $G$. To be specific, $M_\rho = \mathbb H^2 \times \pu /\sim$ where the equivalence relation is defined by $(p,z)\sim (p', z')$ if there exists $g \in G$ such that $(p', z') = (g(p), \rho(g)(z))$.  

 The natural projection $\pi:M_\rho\cv X$  makes $M_\rho$ a flat $\pu$-bundle over $X$, 
 and the horizontal foliation on  
 $\hh^2\times \pu$ descends to 
 $M_\rho$  as a holomorphic foliation with leaves transverse to the $\pu$ fibers, whose monodromy representation is $\rho$.
 
 The point is to study the   
growth rate of holonomy between vertical fibers in $M_\rho$ along generic  Brownian paths  in $X$, 
measured with respect to some smooth family of spherical metrics on the fibers. 
When $X$ is compact the choice of the metrics on   fibers is harmless. 
On the other hand in the non-compact case this requires some care,  so let us dwell a little bit on this point.

 \medskip
 
By a spherical metric, we mean a conformal metric on $\mathbb P^1$ of curvature $1$. 
Equivalently, it is the Fubini-Study metric associated to some Hermitian structure on $\cd$.
Fix a smooth family of spherical metrics on the fibers of  $M_{\rho}$. If  $\omega$ is any continuous path on  $X$, we denote by 
$h_\rho(\omega)$ its holonomy $h_\rho(\omega):\pi^{-1}(\gamma(0))\cv \pi^{-1}(\gamma(1))$. 
We define $\norm{h_\rho(\omega)}$ as $\sup_{x\in \pi^{-1} (\gamma (0) )} \norm{d_x h_\rho(\omega)}$, where the norm is relative to the choice of conformal metrics on the fibers.

We say that the family of spherical metrics is {\em normalized} if the induced  
metric on $\pi^{-1}(\star)$ is the standard Fubini-Study metric. This means  that the 
natural map $\set{0}\times \pu \cv \pi^{-1}(\star)$ induced by the definition of $M_\rho$ as a quotient, is an isometry when $\pu$ is given the standard Fubini-Study metric.  This condition is meant to ensure that  for closed loops, the definition of $\norm{h_\rho(\omega)}$ coincides with the natural one.
 
\medskip

By definition a family of spherical metrics on $M_{\rho}$ is 
\textit{Lipschitz} if  there exists a constant $C_\rho>0$ such that for every smooth path $\omega$ 
on $X$,   it holds that 
\begin{equation}\label{eq:lipschitz}  
\log \norm{h_\rho(\omega)} \leq C_\rho \mathrm{length} (\omega)   ,
\end{equation}
where the length is relative to the hyperbolic metric on $X$. 

When $X$ is compact any smooth family of metrics is of course Lipschitz. In the non compact case we have the following result.

\begin{prop}\label{prop:lipschitz}
If $\rho$ is a parabolic representation of $G$, then there exists a smooth normalized Lipschitz family of spherical metrics on the fibers 
of $M_\rho$. Furthermore the constant $C_\rho$ in \eqref{eq:lipschitz} can be chosen to be locally uniformly bounded on $\mathrm
{Hom}_{\rm par}(G, \PSL)$.
\end{prop}

Notice that the result is false for a general representation. Indeed consider a path of length $L$ turning about $e^L$ times around a cusp of $X$.
Let $g\in G$ representing a simple loop aroung that cusp. 
 Then if $\rho(g)$ is loxodromic,  we see that  \eqref{eq:lipschitz} will be violated. On the other hand it can be shown that the proposition (as well as all results in this section) remains true as soon as   every such $\rho(g)$ has eigenvalues only on the unit circle. 

\begin{proof}
The first observation is that the set of spherical metrics on $\pu$ is parameterized by $\mathbb H^3$. 
This follows from the fact that  
such a metric is the visual metric $\theta(p)$ from a point $p\in \mathbb H^3$ on the sphere at infinity $\pu$. 
Another way to see it is to observe that a spherical metric is defined by a M\"obius map  from $\pu$  
(equipped with the standard Fubini-Study metric) up to precomposition by an isometry, and $\PSL/\mathrm{PSU}(2)\simeq \hh^3$. 

We denote  by $ds_p$ the metric associated to $p\in \hh^3$. 
We claim that if $p,q\in \hh^3$, we have that
 \begin{equation}\label{eq:spsq}
 \log \norm{\frac{ds_p}{ds_q}}_\infty = d_{\hh}(p,q) + O(1).
 \end{equation} 
This follows from an elementary computation. Assume first that $q=0$, 
so that $ds_q^2 = ds_0^2 = 
\frac{\abs{dz}^2}{(1+\abs{z}^2)^2}$ is the usual spherical metric. Let $p \in \hh^3$, $p = A^{-1}(0)$  for some $A\in \PSL$, so that $ds_p	  = A^* ds_0	$. Putting $A = \lrpar{\begin{smallmatrix} a & b 
\\ c & d\end{smallmatrix}}$, we infer that 
$$ds_p	= \frac{\abs{dz}	}{\abs{az+b}^2 + \abs{cz+d}^2},$$  therefore
 $  \norm{\frac{ds_p}{ds_0}}_\infty = \norm{A^{-1}}^2 =\norm{A}^2$, and finally
 $$ \log  \norm{\frac{ds_p }{ds_0 }}_\infty = 2 \log \norm{A} = d_{\hh}(0, p) +O(1).$$ 
To handle the general case, pick an isometry $h$ of $\hh^3$ bringing $q$ back to 0, and observe that since $\frac{ds_p}{ds_q}$ 
is a function on $\pu$, 
$$  \norm{\frac{h^*ds_p }{h^*ds_q }}_\infty  =   \norm{h^* \frac{ds_p }{ds_q }}_\infty=   \norm{\frac{ds_p }{ds_q }}_\infty .$$

\medskip

A smooth family of spherical metrics on a trivial line bundle $D \times \pu$ is now simply the data of a smooth map 
$f:D\rightarrow \mathbb H^3$. A smooth family of spherical metrics on the fibers of $\hh^2\times \pu$ 
descends to a family of spherical metrics on $M_{\rho}$ if and only if  the associated 
 map $f$ is $\rho$-equivariant, i.e. for every $g\in G$, $f\circ  \rho_{\rm can}(g) = \rho(g)\circ f$. 
 Furthermore the family is Lipschitz iff $f$ is, and 
 the constant $C_\rho$ is controlled by the Lipschitz constant of $f$.   This follows 
 from \eqref{eq:spsq}, since  the holonomy of the horizontal foliation on 
  $\hh^2\times \pu$ is the identity.

%
  
\medskip

To prove the proposition it is thus enough to show that there exists a smooth Lipschitz $\rho$-equivariant map $f: \mathbb H^2 \rightarrow \mathbb H^3$.  The family of spherical metrics will be normalized provided $f(0) = 0$.  
To construct such an $f$,  recall that since
 $X$ is of finite type, there exists a finite number of open disjoint subsets $C_1,\ldots, C_r$ of $X$ such that 
$X\setminus \bigcup _i C_i $ is compact, and each $C_i$ is isometric to the quotient of $\{ z\in \mathbb H^2 \ |\ \Im z > \rho_i\}$.
  by the translation $z\mapsto z+1$ for some positive $\rho_i$. For every $i=1,\ldots,r$, we pick some lift $\widetilde{C_i}\subset \mathbb H^2$ of $C_i$, and we first define $f$ on the closure of $\widetilde{C_i}$. Denote by $g_i\in G$ a generator of the stabilizer of $\widetilde{C_i}$ in $G$. By assumption, $\rho(g_i)$ is parabolic, hence conjugate  to $g_i$. Thus, there exists an isometric immersion $f: \overline{\widetilde{C_i}} \rightarrow \mathbb H^3$ such that $f\circ g_i = \rho(g_i) \circ f$ (we take the map in $\mathrm{PSL} (2,\mathbb C)$ which conjugates $g_i$ to $\rho(g_i)$). We extend the map $f$ by $\rho$-equivariance to a map defined on $ \bigcup_i G \overline{\widetilde{C_i}}$. Then, because $\mathbb H^3$ is contractible, a classical topological argument shows that the map $f$ can be extended as a smooth $\rho$-equivariant map from $\mathbb H^2$ to $\mathbb H^3$ such that $f(0)=0$. We claim that this map is Lipschitz. Indeed, $\norm{Df}$ is $G$-invariant (the norm of $Df$ is measured using hyperbolic metrics on source and target), thus descends to a function on $X$. But since $f$ is an isometric embedding on $\widetilde{C_i}$, we have $\norm {Df} = 1$ on the cusps $C_i$'s. Hence $\norm{Df}$ is bounded, since $X\setminus \cup_i C_i$ is compact, and $f$ is indeed our desired 
  smooth  Lipschitz $\rho$-equivariant map.
Finally, it is clear that the Lipschitz constant of $f$ can be chosen to be locally uniformly bounded in 
$\mathrm{Hom}_{\rm par}(G, \PSL)$, so the same is true of the constant $C_\rho$. 
\end{proof}

The following consequence of the previous proposition will be useful.

\begin{cor}\label{cor:compar} 
 For every parabolic representation $\rho$
there exists a constant $\beta_\rho$ depending only on the constant $C_\rho$ of \eqref{eq:lipschitz} 
such that  $\norm{\rho(g)}\leq  \norm{\rho_{\rm can} (g)}^{\beta_\rho}$.  
\end{cor}

\begin{proof}
 Indeed for $g\in G$, let $\sigma_g$ be the geodesic segment $[0, \rho_{\rm can}(g)0]$ in $\hh^2$. Then in $M_\rho$, the holonomy of the projection of  $\sigma_g$ in $X$ is $\rho(g)$.  By 
 \eqref{eq:norm distance}, 
 \eqref{eq:lipschitz} and the normalization of the family of spherical metrics
   we infer that 
$$\log \norm{h_\rho({\sigma_g})} = \log \norm{\rho(g)}  \leq C_\rho \length(\sigma_g) \leq C_\rho 
\argcosh \frac{\norm{\rho_{\rm can}(g)}^2_2}{2}\leq C_\rho \log\lrpar{2\norm{\rho_{\rm can}(g)}^2}.$$ 
Here we use the inequalities $\norm{\gamma}^2_2\leq 2 \norm{\gamma}^2$ and $\argcosh x\leq \log 2x$. Hence we get that $\norm{\rho(g)}\leq (2\norm{\rho_{\rm can}(g)})^{2 C_\rho}$. To conclude the proof, just note that since $\rho_{\rm can}(g)$ contains  
no elliptic transformations, 
 there exists $c>0$ such that 
$\norm{\rho_{\rm can}(g)}>1+c$, so 
there exists $\alpha$ such that $2\leq   \norm{\rho_{\rm can}(g)}^\alpha$. 
The result follows.  
\end{proof}

We have the following  analogue of Definition \ref{def:lyap} in this context. 

 \begin{prop}\label{prop:lyap}
Let $X$ be a hyperbolic Riemann surface of finite type, $\rho:\pi_1(X)\cv \PSL$ be a parabolic representation, and $\norm{.}$ a smooth normalized  Lipschitz  family of spherical metrics on $M_{\rho}$. With notation as above, for  $W$-a.e. path  $\omega$, the limit 
 \begin{equation}\label{eq:proplyap}
 \chi  (\rho; \omega) = \lim_{t\cv\infty} \unsur{t}\log\norm{h_\rho\big({\omega\rest{[0, t]}}\big)}
 \end{equation}
  exists and does not depend on  $\omega$ nor on the the choice of $\norm{\cdot}$.  
 \end{prop}

 \begin{proof}
 Let $\sigma$ be the shift semi-group acting on $\Omega$ by $\sigma_s:\omega(\cdot) \mapsto \omega(\cdot+s)$.  Recall that $W=\int W_x dx$. Since the Riemannian measure is invariant under heat flow,  $W$  is invariant under the time shift $\sigma_s:\omega(\cdot) \mapsto \omega(\cdot+s)$. The family of functionals $H_t(\omega):\omega\mapsto \log\norm{h_\rho\big({\omega\rest{[0, t]}}\big)}$ defines a sub-additive cocycle. In order to apply Kingman's subadditive ergodic theorem, we need to check that this cocycle is $W$-integrable, that is, that 
 $\int  \log\norm{dh_\rho\big({\omega\rest{[0, t]}}\big)} dW(\omega)<\infty$ for some (hence every)
 $t>0$. 
For this we use the  Lipschitz property of the family of spherical metrics. Let  $\widehat{X}$ be a Dirichlet fundamental domain containing $0$
  for the action of $\Gamma$ on $\hh^2$. For a Brownian path $\omega$, let $\widetilde{\omega}$ be its lift to $\hh^2$ starting from the    lift of $\omega(0)$  lying in   $\widetilde{X}$. By \eqref{eq:lipschitz} we have that 
 \[  \log \norm{h_\rho\big({\omega_{|[0,t]}}\big)} \lesssim d_{\hh} (\widetilde \omega(0), \widetilde\omega (t) ) + 
d_{\hh}(0,\widetilde  \omega(0)).
\]
Thus 
$$\int  \log\norm{dh_\rho\big({\omega\rest{[0, t]}}\big)} dW(\omega) \lesssim 
\int_{\widetilde X}  \ee_x\lrpar{ d_{\hh} (\widetilde \omega(0), \widetilde\omega (t) ) } dx
$$ which is finite thanks to 
the (super)exponential decay of the heat kernel on hyperbolic plane (cf. \eqref{eq:davies}).
\end{proof}
 
The following result  finally justifies the validity of Definition \ref{def:lyap}.

\begin{prop}\label{prop:lyap2}
Let $X$, $\rho$, and $\norm{\ }$ be as in Proposition \ref{prop:lyap}. Then for $W_\star$ a.e. Brownian path $\omega$,  we have that $$\lim_{t\cv\infty} \unsur{t}\log \norm{\rho(\widetilde \omega_t)} = \lim_{t\cv\infty} \unsur{t}\log \norm{h_\rho\big(\omega\rest{[0, t]}\big)} = \chi_{\rm Brown}(\rho).$$ 
 Furthermore the limit in  \eqref{eq:proplyap} is also equal to $\chi_{\rm Brown}(\rho)$.
 \end{prop}
 
 \begin{proof}
 Let us first show that for $W_\star$ a.e. $\omega$, $\lim_{t\cv\infty} \unsur{t}\log \norm{h_\rho\big(\omega\rest{[0, t]}\big)}$ exists and has the 
 same value as the one in \eqref{eq:proplyap}. Indeed the diffusion of the Dirac mass $\delta_\star$ at time 1 is absolutely continuous 
 with respect to volume measure on $X$. Now we simply write $h_\rho\big(\omega\rest{[0,t]}\big) = h_\rho\big(\omega\rest{[1,t]}\big)\circ h_\rho 
 \big(\omega\rest{[0,1]}\big)$  and by the Markov property of Brownian motion together with Proposition \ref{prop:lyap}, we get that 
  for $W_\star$ a.e. $\omega$, $\unsur{t}\log \norm{\omega\rest{[1, t]}}$ converges to the limit in \eqref{eq:proplyap}.
 
 \medskip
 
To finish the proof, we show that for   $W_\star$ a.e. $\omega$, 
\begin{equation}\label{eq:omega tilde}
 \unsur{t} \abs{\log \norm{\rho(\widetilde \omega_t)} - \log \norm{h_\rho\big(\omega\rest{[0, t]}\big)}} \underset{t\cv\infty}
\longrightarrow 0.
\end{equation}
If $X$ is compact this is obvious since $d_X(\omega(t), \star)$ is uniformly bounded. In 
the general case, by the Lipschitz property we have to show that for $W_\star$ a.e. $\omega$, $d_X(\omega(t), 
\star) = o(t)$. A way to see this is to use the following two facts: first, a 
theorem of Sullivan \cite{sullivan disjoint spheres} asserts that for almost every $v$ in the unit tangent 
space $S_\star X$ to $\star$, the unit speed   geodesic ray $\gamma(t)$ issued from $\star$ pointing in the direction 
$v$ satisfies 
$d_X(\star, \gamma(t)) = O(\log t)$. The other  result we use is the  classical fact that 
Brownian paths on hyperbolic space  are shadowed by   geodesic rays (see e.g. \cite{ancona}). More precisely, if $\omega$ is 
a Brownian path  starting from $0 \in \hh^2$, then $\lim_{t\cv\infty} \omega(t) = \omega_\infty$ exists in $\fr\hh^2$, and 
if $\gamma_{x,\omega_\infty}$ denotes the unit speed   geodesic ray joining $x$ and $\omega_\infty$, then
 for $W_0$ a.e. $\omega$, 
$d_{\hh}(\omega(t), \gamma_{x,\omega_\infty} (  t ) = o(t)$ (this is actually $O(\log t)$ 
by a theorem of Ancona \cite{ancona}).  
Note that the distribution of $\omega_\infty$ is absolutely continuous on the circle at infinity (with density given by the 
Poisson kernel). 
Therefore on $X$, for $W_\star$ a.e. $\omega$, there exists a unit tangent vector such that $d_X\lrpar{\omega(t), 
\gamma_{\star,v}(  t) }= o(t)$. Since the distribution of the tangent vector $v$ is induced by that of $\omega_
\infty$, it is absolutely continuous with respect to Lebesgue measure on $S_\star X$. So we can  combine this estimate with Sullivan's, and conclude that $d_X(\omega(t), \star) = o(t)$ a.s., 
 which implies \eqref{eq:omega tilde}. 
\end{proof}

For further reference, we now make the easy observation that the natural Lyapunov exponent is insensitive to finite coverings. 

\begin{prop}\label{prop:covering}
Let $X$ be a hyperbolic Riemann surface of finite type, and  $\pi: (Y,  y_0)\cv(X, \star)$ be a finite covering. Any 
  parabolic representation   $\rho : \pi_1(X, \star) \cv\PSL$ induces by restriction a parabolic   representation $\rho_Y$
  of $\pi_1(Y, y_0)$. Then, denoting by  $\chi_{\rm Brown}(\rho_Y, Y)$ the Lyapunov exponent of $\rho_Y$ associated to the Riemann surface structure on $Y$, we have that $\chi_{\rm Brown}(\rho,X)=\chi_{\rm Brown}(\rho_Y, Y)$.
\end{prop}

\begin{proof}
This is just a matter of following step by step the proof of the existence  of $\chi_{\rm Brown}(\rho, X)$. Indeed, we observe that 
the suspension $Y_{\rho_Y}$ is naturally a covering of $M_\rho$, and the lift of a Lipschitz metric $\norm{\cdot}$ 
on the fibers of $M_\rho$ is a Lipschitz metric  on the fibers of $Y_{\rho_Y}$. As already observed, the Wiener spaces 
$W_\star$ on $X$  and $W_{y_0}$  are naturally identified to  $W_0$ in $\hh^2$. Now by construction, 
if $\omega$ is a Brownian path issued from $y_0$ in $Y$, we have that $\norm{h_{\rho_Y}(\omega\rest{[0,t]})} = 
\norm{h_\rho(\pi(\omega)\rest{[0,t]})}$. The result follows. 
\end{proof}

\medskip

Let now $(\rho_\la)_{\la\in \La}$ be a holomorphic  family of parabolic  representations of $G$. Thanks to
  Definition \ref{def:lyap} we can define a Lyapunov exponent function $\chi_{\rm Brown}$ on $\La$, 
and we put $\tbif = dd^c\chi_{\rm Brown}$. We will refer to it as the \textit{natural bifurcation current} of the family.  The following result asserts that $\tbif$ is indeed a bifurcation current. 
 
 \begin{thm}\label{thm:support brownien}
 Let $X$ be a hyperbolic Riemann surface of finite type, and 
   $(\rho_\la)_{\la\in \La}$ be a holomorphic family of representations of $G=\pi_1(X)$,  satisfying (R1, R2, R3'). Then $\chi_{\rm Brown}$ is a nonnegative  psh function on $\La$, and the positive closed current $\tbif = dd^c\chi_{\rm Brown}$ satisfies $\supp(\tbif) = \bif$.
   
 If furthermore (R3) holds then $\chi_{\rm Brown}$  is positive and locally H\"older continuous.
    \end{thm}
 
The proof consists in 
exhibiting a measure $\mu$ on $G$ satisfying $(M)$ and such that 
$T_{\rm bif, \mu} = \tbif$ (up to some multiplicative constant). It  then suffices to apply Theorem \ref{thm:support}. 
It  will be carried out in the next subsections. 
  
 \begin{rmk}\label{rmk:clock}
We have constructed the natural bifurcation current  by using the hyperbolic metric on  $X$. Let us  observe 
 that when $X$ is compact any other conformal metric would essentially yield the same object. Indeed, choosing a different conformal metric  on $X$ only changes the ``clock'' of Brownian motion, that is, sample Brownian paths for the new metric   
are time-reparameterized Brownian paths relative to the original one. It then follows from 
 an easy ergodicity argument  that the associated Lyapunov exponent function is a constant multiple of $\chi_{\rm Brown}$.  When $X$ has punctures, the same statement holds, provided one restricts to an appropriate class of conformal metrics, for instance the ones that coincide with the Poincar\'e metric outside some compact set of $X$. 
\end{rmk}

\subsection{Discretization of the Brownian motion}\label{subs:discretization}
We will rely on a discretization procedure for the Brownian motion which  was introduced by Furstenberg \cite{furstenberg 
discretization}, and subsequently studied by several authors, in particular Lyons and Sullivan \cite{lyons sullivan} 
(see also Kaimanovich \cite{kaimanovich discretization entropy}, Ballman-Ledrappier \cite{ballmann ledrappier}, Ancona \cite{ancona}).  Since we will need precise moment estimates, 
we include a detailed treatment of the discretization, which  is close (but not identical to) that of Lyons and Sullivan.

\medskip

Let $\Gamma$ be a finite covolume lattice in $\text{PSL}(2,\mathbb R)$, and $X = \Gamma\setminus\mathbb{H}$. Let $0\in \mathbb H$ and $R>0$ such that the balls $ B (\gamma 0, R)$ are disjoint, for $\gamma \in \Gamma$. 
Since the Brownian motion on $X$ is recurrent \cite{grigoryan},
for $0<r\leq R$ the set $\partial _r = \bigcup _{\gamma} \partial B(\gamma 0, r)$ is recurrent, namely a.e. Brownian path starting at a point $x
$ hits $\partial _r$ in finite time. We denote by $T_r$ (resp. $T_R$) the first hitting time to $\partial _r$ (resp. $\fr_R$). Since $\fr_r$ (resp.  $\fr_R$) 
is closed, $T_r$ (resp. $T_R$) is a stopping time.

\subsubsection{Sub-balayage}
Recall that the balayage of a finite measure $m$ on a closed 
recurrent set $E$ is defined to be  the image of the measure $\int W_x dm(x)$ under the map $\omega \mapsto \omega (T_E)$ where $T=T_E(\omega)$ is the first moment where $\omega$ hits $E$.

We denote by $B_\rho$, $S_{\rho}$ respectively the ball and   sphere centered at $0$ and of radius $\rho$. If $y$ is a point interior to $B_R$, the balayage of the Dirac mass at the point $x$ on $S_R$ is a probability measure absolutely continuous with respect to the spherical measure $ds$ on $S_R$ (conveniently normalized so that $ds $ is a probability measure), the density being given by the Poisson kernel. Hence, there is a constant $0 < p < 1$, such that if $x$ is chosen in the sphere $S_r$, its balayage on $S_R$ is bounded from below by $p ds$. 

The following lemma will allow us to select a ``good" set of paths from $S_r$ to $S_R$. 
For every $x\in S_r$, a.e. Brownian path issued from $x$ intersects $S_R$.  The hitting time is 
 $T_R(\omega)$    and we denote by   $\mathrm{e}(\omega) = \omega(T_R(\omega))$  the exit point. We use the standard notation $(\mathcal{F}_t)_{t\geq 0}$ for the Brownian filtration.

\begin{lem}\label{lem:barrier}
For every $x\in S_r$ there exists a $\mathcal{F}_{T_R}$-measurable set $G_x$, with $\pp(G_x) =p$ 
such that  $ \mathrm{e}_*( 1_{G_x}W_x)  =pds$.  

In addition, reducing $p$ slightly if necessary,  $G_x$ can be chosen so that there exists a constant $M$
such that for every $x\in S_r$ and  $\omega\in G_x$, $T_R(\omega)\leq M$.
\end{lem}

\begin{proof} 
Let us first establish  the lemma without the bound $M$ on the exit time. 
Denote by  $b_x(y)ds(y)$   the balayage measure of $\delta_x$ on $S_R$. Notice that $b_x(y)> p$ for every $y\in S_R$.  
For $0\leq t\leq 1$ consider the piece of circle
 $$C_t = \set{\frac{R+r}{2} e^{i2\pi \theta}, 0\leq \theta\leq 1-t} .$$
We use $C_t$ as a ``continuously opening  barrier" for Brownian paths going from $x$ to $S_R$.  For $0\leq t\leq 1$, let $E_t\subset \om_x$ be the set of paths starting at $x$   that avoid $C_t$ before hitting $S_R$. For $t=0$ the barrier is closed so $E_t=\emptyset$ while for $t=1$ it is open and $W_x(E_t)=1$; observe also that for $t<t'$, $C_t\supset C_{t'}$ so $E_t\subset E_{t'}$. 
Therefore,
 for every $y\in S_R$, $f_y: t\mapsto W_x(E_t\vert \mathrm e^{-1}(y))$ is a continuous   increasing function with $f_y(0)=0$ and $f_y(1)=1$. The continuity of $f_y$ follows from the fact that   the set of Brownian paths (starting from $x$ and exiting at $y$),   intersecting 
 $E_{t'}\setminus E_t$ has small measure when $t'-t$ is small. 
 Let $t(y)\in (0,1)$ be the unique value of $t$ such that $f_y(t(y)) = \frac{p}{b_x(y)}$.  We can now define
  $$G_x = \set{\omega\in \om_x, \ \omega\rest{[0, T_R(\omega)]}\cap C_{t(\mathrm{e}(\omega))} =\emptyset}.$$ Clearly, this set is $\mathcal{F}_{T_R}$ measurable (recall that $\mathrm{e}(\omega) = \omega(T_R(\omega))$). 
  Since by definition 
  $G_x\cap \mathrm e^{-1}(y) = E_{t(y)} \cap \mathrm e^{-1}(y)$ we infer that 
  \begin{align*}
 \mathrm e_*(1_{G_x}W_x) &=\mathrm e_* \int W_x(G_x\vert \mathrm{e}^{-1}(\mathrm e(\omega))) dW_x(\omega)
   =\mathrm e_* \int W_x(E_{t(\mathrm e(\omega)))}\vert \mathrm{e}^{-1}(\mathrm e(\omega))) dW_x(\omega)\\
&=   W_x(E_{t(y)}\vert \mathrm{e}^{-1}(y)) \; \mathrm e_*W_x 
   =W_x(E_{t(y)}\vert \mathrm{e}^{-1}(y)) b_x(y)ds(y)
   = pds,
 \end{align*}  
   which was the desired property. 
   
\medskip

To prove the second claim, observe  that for every $p_1<p$, if 
$A_x(M)\subset W_x$ denotes the set of paths such that $T_R(\omega)\leq M $, then for $M$ 
sufficiently large, $\mathrm{e}_*(1_{A_x(\tau_1)} W_x)\geq p_1ds$, uniformly in $x$ --this is a 
consequence of Dini's Theorem. From this point, to get the result it suffices 
to repeat the above proof  with 
$1_{A_x(\tau_1)} W_x$ instead of $W_x$.   
   \end{proof}

We extend by equivariance the definition of the sets $G_x$ to all of $\partial _r$ by setting $G_{\gamma x} = \gamma G_x$.

\subsubsection{A stopping time}\label{sss:stopping time}
For every point $x_0 \in \mathbb H$ and every $\omega\in \om_{x_0}$, 
we define a sequence of stopping times 
$$U_0 =T_R(\omega)< u_1= u_1(\omega)< U_1 = U_1(\omega)< \ldots < u_n(\omega) < U_n (\omega) <\ldots $$ 
as follows:   $U_0=T_R(\omega)$, and by induction, $u_n$ is the first time after $U_{n-1}$ at which $\omega$ hits $\partial _r$, that is $u_n(\omega) = U_{n-1} (\omega) + T_r(\sigma_{U_{n-1}}(\omega))$, where as before $\sigma_t$ denotes time shift by $t$. Likewise, 
 $U_n$  is the first time after $u_n$ at which $\omega$ hits $\partial_R$, i.e. $U_n (\omega) = u_n(\omega) + T_R (\sigma_{u_{n}}(\omega)) $. 
 
By the strong Markov property of Brownian motion, almost surely there exists an integer $n$ such that $\sigma_{u_n} \omega$ belongs to $G_{\omega(u_n)}$. We define $T (\omega) := U_k (\omega)$, where $k=k(\omega)$ is the first such integer. The measurability property of the sets 
$G_x$  ensures that   $T$ is a stopping time (intuitively at time $t$ we know if we already went from 
$\fr_r$ to $\fr_R$ through some $G_x$). We   let $x_1(\omega) =\omega(T(\omega))$ and $r_1 (\omega)$ be the element of $\Gamma$ such that $x_1 \in \partial B(r_1\cdot 0,R)$. The law of $x_1$ defines a transition probability $\Pi(x_0, \cdot)$.  By construction it satisfies  the equivariance property $\Pi(\gamma x_0, \cdot) =\gamma_*\Pi(x_0, \cdot)$ for $\gamma\in \Gamma$. 

Repeating this process independently  gives rise to sequences $(x_n)_{n\geq 0}\in (\mathbb H^2)^\nn$ and $(r_n)_{n\geq 1}\in \Gamma^\nn$.  
 
 \subsubsection{The discretization measure} 
 It follows from the construction that 
 whatever the distribution of the point $x_0$ is, the distribution of the point $x_1$ is a convex 
 combination of the spherical measures $ds_{\gamma} = \gamma_* ds$ on $\partial _R$. Hence, if 
 $x_0$ has distribution $ds$ on $S_R$ (observe that $U_0=0$ in this case), then $x_1$ has 
 distribution $\mu \ast ds$ for a certain measure $\mu$ on $\Gamma$. 
 
 We now define a Markov chain $(x_n)$ with initial distribution given by the normalized spherical measure $ds$ on $S_R$ and transition kernel $\Pi$. By the equivariance property of $\Pi$, the distribution of $x_n$ is $\mu^n \ast ds$. 
 
  Furthermore, since the transition kernel  is defined by  stopping the Brownian motion, 
  by the strong Markov property we infer that 
  the map $\omega \mapsto (r_n(\omega) ) $ sends the measure $\int_{S_R} W_x ds(x)$ to the path measure $\mathsf P$ on $\Gamma^\nn$ of the right 
   random walk on $\Gamma$ induced by $\mu$. 
  
By definition  $\mu$ is the discretization measure (notice that it is not unique). 

\subsection{The cocompact case} \label{subs:compact}
The core of the proof of Theorem \ref{thm:support brownien} is the fact
 that the discretization measure satisfies the exponential moment condition (M). When $X$ is compact, this 
fact is well-known to the specialists (for instance it is announced without proof in \cite{ancona}), but also not so easy to find in print.  Let us explain this case first.
Recall that $\Gamma = \rho_{\rm can}(G)$. 

\begin{prop}\label{prop:cocompact}
If the lattice $\Gamma$ is cocompact, then the discretization measure $\mu$ satisfies the exponential moment condition \eqref{eq:exponential moment in the group} with respect to word length in $G$. 
\end{prop}
 
 \begin{proof}
By the stronger  assertion of   Lemma \ref{lem:barrier}, for every $p_1<p$  there exists a deterministic time $\tau_1$ such that for every $x\in S_r$ there exists a
  $\mathcal{F}_{T_R}$ measurable set $G_x$ such that $\mathrm{e}_*(1_{G_x}W_x) =p_1ds$ and moreover for every $\omega\in G_x^1$, $T_R(\omega)\leq \tau_1$.  
  
On the other hand, since $X = \Gamma\setminus \mathbb H$ is compact, there exists $0<p_2<1$ and $\tau_2>0$ such that for every $x\in \mathbb{H}$, the probability that a  Brownian path starting at $x$ hits $\partial _r$ in time less than $\tau_2$ is bounded from below by $p_2$. 

Therefore, by the  Markov property we infer that for every $x\in \mathbb H$ and every  $t_0$, 
$$\mathbb P_x ( T\geq t_0 + \tau_1 + \tau _2 ) \leq (1-p_1 p_2) \cdot \mathbb P_x ( T \geq t),$$ 
 where $T$ is the stopping time defined in \S \ref{sss:stopping time}. From this we deduce that there exist constants $C,\alpha >0$ independent of $x$ such that 
 \begin{equation}\label{eq:exp decay}  \mathbb P_x ( T \geq t ) \leq C  e^{-\alpha t}.
  \end{equation}
 Since the discretization measure is the image of $\int_{S_R} W_x ds(x)$ under $\omega\mapsto \omega(T(\omega))$, and the  distance between $r_1(\omega) \cdot 0$ and $\omega(T(\omega))$ is $R$,
  for every $\delta >0$,  we have that 
$$ \mu (\gamma\in \Gamma, \   d(\gamma 0, 0) \geq \delta) \leq \int \mathbb P_{x} ( d ( \omega (T(\omega)), 0)  \geq \delta -R )\ ds(x) ,$$ where $d(\cdot, \cdot)$ denotes hyperbolic distance.   Fix $K>1$   as in Lemma \ref{lem:martingale} below.   
For every $x\in S_R$, we have that 
\begin{align*}
\mathbb P_{x} ( d ( \omega (T) , 0) \geq \delta - R) 
&\leq \mathbb P_x (d( \omega (T) , x) \geq \delta- 2R )\\
 &\leq  \mathbb P_x \left( \sup _{t\leq (\delta-2R)/K} d(\omega(t), x) \geq (\delta-2R)\right) + 
 \mathbb P_{x} \left(T \geq \frac{\delta-2R}{K} \right) .  
\end{align*}
By Lemma \ref{lem:martingale} below the first term decreases exponentially fast when $\delta$ tends 
to infinity, while  the second does because of  \eqref{eq:exp decay}. By the homogeneity of hyperbolic 
plane, these estimates are uniform for $x\in S_R$, and me conclude that 
 $\mu (\gamma\in \Gamma, \   d(\gamma (0),0) \geq \delta ) $ decreases exponentially with 
 $\delta$. Finally, since $\Gamma = \rho_{\rm can}(G)$ is cocompact, 
 $(\Gamma0, d)$ is quasi-isometric to any Cayley graph of $G$ 
and we are done. 
\end{proof} 

The following result is presumably well-known but we could not locate a precise reference.

\begin{lem}\label{lem:martingale}
There exists a positive constant $K$ such that 
$$\pp_0\lrpar{\sup_{0\leq s\leq t} d(0, \omega(s))\geq Kt}$$ decreases exponentially with $t$.
\end{lem}

\begin{proof}
Let us work in the disk model of hyperbolic plane. The expression for the hyperbolic 
distance is $\exp (d(0, z)) =  \frac{1+\abs{z}}{1-\abs{z}}$.   Since $x\mapsto \frac{1+x}{1-x}$ is convex and increasing, 
 $z\mapsto \exp (d(0, z)) $ is a subharmonic function. From this we infer that
 if $\omega(t)$ is the Brownian motion starting at 0, the process defined by 
 $\big(\exp (d(0, \omega(t)))\big)_{t>0}$ is a positive submartingale. 
 
From the asymptotic estimate for the heat kernel given in \S \ref{subs:brownian}, we infer that 
$\ee_0\big(\exp (d(0, \omega(t)))\big)$ is finite for every $t$. 
Now we have that
\begin{align} \label{eq:sequence}
\ee_0\big(\exp (d(0, \omega(t)))\big)&\leq \ee_0\big(\exp (d(0, \omega(\lceil t\rceil)))\big) \text{ by the submartingale property}\\
&\leq \ee_0\lrpar{\exp \lrpar{\sum_{i=0}^{\lceil t\rceil-1} d(\omega(i), \omega(i+1))   }} \text{ by the triangle inequality} \notag \\
 &= \ee_0\big(\exp (d(0, \omega(1)))\big) ^{\lceil t\rceil} \text{ by the Markov property} \notag \\
 &\leq \ee_0\big(\exp (d(0, \omega(1)))\big) ^{t+1} \notag
 \end{align}
  (for the equality on the third line we take successive conditional expectations and use the homogeneity of the hyperbolic plane).
The Doob martingale inequality applied to $\exp (d(0, \omega(t)))$ asserts that  for every $\alpha>0$, 
\begin{equation}\label{eq:doob}
 \pp_0\lrpar{\sup_{0\leq s\leq t} \exp d (0, \omega(s))\geq \alpha}\leq 
 \unsur {\alpha} \ee_0\big(\exp (d(0, \omega(t)))\big)
\end{equation}
 Hence if $K$ is any real number larger than 
 $\ee_0\big(\exp (d(0, \omega(1)))\big)$, applying  inequality \eqref{eq:doob} with $\alpha = \exp(Kt)$  together 
with \eqref{eq:sequence}
finishes the proof.
\end{proof}

\begin{rmk}
It is likely that the result holds for any $K>1$. A possible way to achieve this would be to refine the subdivision of the interval $[0,t]$ in \eqref{eq:sequence}.
\end{rmk} 

\subsection{The finite volume case}\label{subs:proof}
  We will need the following estimate on the  hitting time of the Brownian motion 
 on a given subset of $X$.
 We thank J. Franchi for explaining it to us.  

\begin{prop}\label{prop:franchi}
Let $X$ be a hyperbolic  Riemann surface of finite type, endowed with its Poincar\'e metric, and 
Liouville measure $dx$. Let $K\subset X$ be a compact set of positive area and define a stopping 
time for the Brownian motion by 
$$T_K(\omega) = \inf\set{t >0, \ \omega(t)\in K}.$$ Then there exists a constant $\alpha>0$ such that  for every 
$x\in X$, $\ee_x(e^{\alpha T_K})<+\infty$. Furthermore, the function $x\mapsto  \ee_x(e^{\alpha T_K})
<+\infty$ is locally bounded and in $L^2(X,dx)$. 
\end{prop}

Of course this uniformity is highly non-trivial for some  paths   
 go deeply inside a cusp before reaching $K$.

\begin{proof}
Let $(\Pi_t)_{t\geq 0}$ be the Brownian semigroup  on $X$, associated to $  \Delta$.     
The spectral theory of the Laplacian on $X$  implies that 
 $\Pi_t$  is strongly
  contracting, specifically there exists a constant $c>0$ such that for every $f\in L^2(X)$ of   zero mean, 
$\norm{\Pi_tf}_{L^2}\leq e^{-ct} \norm{f}_{L^2}$ (see e.g. \cite{franchi le jan} for 
 an elementary approach to the Poincar\'e inequality on $X$, which classically implies a positive lower bound for the first eigenvalue of the Laplacian). Therefore the family of operators defined by 
$Q_T =\unsur{T}\int_0^T \Pi_t$ converges strongly to $f\mapsto \int_X f=\int f(x)dx$ in $L^2$. More precisely there exists a function $c(T)$ with $\lim_{+\infty} c(T) =0$ such that 
$$\norm{Q_Tf - \int_X f}_{L^2} \leq c(T)\norm{f}_{L^2}.$$ Indeed write $f=f_0 + \int f$ with   
$\int f_0=0$. Notice that $\norm{f_0}_{L^2}\leq 2\norm{f}_{L^2}$. Then we have that  
\begin{align*}
\norm{Q_Tf - \int f}_{L^2} & =
\norm{Q_Tf_0}_{L ^2} = \norm{\unsur{T} \int_0^{T_0} \Pi_t f_0+ \unsur{T}\int_{T_0}^T \Pi_t f_0}   
\leq  \frac{2}{T}\norm{f}_{L^2} + \frac{T-T_0}{T} e^{-cT}\norm{f}_{L^2}.
\end{align*}
In this setting, a theorem by   Carmona and Klein   \cite[Thm 1]{carmona klein}  asserts that 
$x\mapsto  \ee_x(e^{\alpha T_K})$ belongs to $L^2(X)$, in particular it is finite a.e. 
Furthermore, elaborating on  \cite[Rmk 1]{carmona klein}, we  see that $\ee_x(e^{\alpha T_K})$ is locally bounded. Indeed let $t_0>0$, and use the Markov property to write 
\begin{align*}
\ee_x(e^{\alpha T_K})  = \int_{\set{T_k\leq t}} e^{\alpha T_K} dW_x &+ \int_{\set{T_k> t_0}} 
e^{\alpha T_K} dW_x 
 \leq e^{\alpha t_0} + e^{\alpha t_0} \int  \ee_y(e^{\alpha T_K}) p_{t_0}(x,y)dy\\
&\leq  e^{\alpha t_0} + e^{\alpha t_0} \norm{\ee_y(e^{\alpha T_K})}_{L^2} 
\norm{p_{t_0}(x, \cdot)}_{L^2},
\end{align*}
where $\Pi_t(x, y)dy$ is the transition kernel.
 Now, as it easily follows from the estimates  in \S\ref{subs:brownian}, for fixed $t_0$ the heat kernel $p_{t_0}(x, \cdot)$
is in $L^2$, locally uniformly with respect to $x$. The result follows.
\end{proof}

From this proposition we get estimates on the stopping times $T_r$ and $T_R$ in $\mathbb{H}$.

\begin{cor}\label{cor:franchi}
There exists uniform constants $\alpha_0$ and $M$ such that for every $x\in \fr_R$, 
$\ee_x(\exp{\alpha_0 T_r})\leq M$  (resp. for every $x\in \fr_r$, 
$\ee_x(\exp{\alpha_0 T_R})\leq M$).
\end{cor}  

\begin{proof}
To prove the corollary we work on $X$. Abusing notation we set 
$\fr_r = \pi S_r$, and similarly for $\fr_R$, where $\pi:\hh^2\cv X$ is the natural projection. Recall that by assumption $B(0, R)$ projects to a ball on $X$. 
Assume that  $x\in \fr_R$, and fatten $\fr_r$ on the side opposite to $x$ to make it a 
closed annulus $\widetilde{\fr_r}$.  Then the hitting times of $\fr_r$ and $\widetilde{\fr_r}$ are the same, and we conclude by Proposition \ref{prop:franchi}. The other case is similar.  
\end{proof}

The next step is to show that the stopping time  $T$ of \ref{sss:stopping time} defining the discretization
 admits an exponential moment.
 
 \begin{prop}\label{prop:exponential moment stopping time} 
With notation as in \S \ref{subs:discretization}, 
 there exist constants  $\alpha$ and $M$ such that for every $x\in S_R$, 
 $$\ee_x(\exp{(\alpha T)})\leq M.$$
 \end{prop}
 
 \begin{proof}
Let $x\in S_R$. Recall that for a.e. $\omega\in W_x$ 
we defined in \S\ref{sss:stopping time} the integer $k(\omega)$ as the first time at which $\omega(u_n(\omega))$ enters the good set $G_{\omega(u_n)}$.
For a constant $\alpha$ to be determined later, 
we decompose the integral $\int \exp (\alpha T) dW_x$ as a sum 
\[ \sum_K \int _{\Omega_x^k} \exp (\alpha T ) dW_x \]
where $\Omega_x^K = \{  \omega\ |\ k(\omega) = K \}$. 
We define a sequence of  shifted paths    $\omega_l$, $0\leq l\leq 2K-1$ 
 by $\omega_0 = \omega$, $\omega_1 = \sigma_{u_1}(\omega)$, $\omega_2 = \sigma_{U_1}(\omega)$, etc. 
 We also define for $1\leq l\leq 2K-1$, $y_l(\omega) = \omega_l(0)$, so that for even $l$ (resp. odd $l$), 
 $y_l\in S_R$ (resp. $y_l\in S_r$). Notice that by definition of $K$, 
 for $l = 2j-1$, $j<K$, $\omega_l\notin G_{y_l}$ while 
 $\omega_{2K-1}\in G_{y_{2K-1}}$.
 
Writing 
 $T(\omega) = T_r(\omega_0)+T_R(\omega_1)+\cdots + T_R(\omega_{2K-1})$,  
 taking successive conditional expectations with respect to the events $T_{\bullet}(\omega_l) = y_{l+1}$ (with $\bullet=r$ or $R$) and 
   applying the strong Markov property  yields
    \begin{align} \label{eq:integral} 
 &    \int_{\Omega_x^k} \exp (\alpha T ) dW_x = \int_{\Omega_x} \exp (\alpha T_r(\omega_0)) dW_x (\omega_0) \cdots  \\
     & \cdots  \int _{G_{y_{2j-1}} ^c}  \exp (\alpha T_R (\omega_{2j-1}) )dW_{y_{2j-1}}(\omega_{2j-1}) \int _{\Omega_{y_{2j}}} \exp (\alpha T_r(\omega_{2j}))dW_{y_2j}(\omega_{2j}) \cdots \notag\\
&\quad \cdots \int_{G_{y_{2K-1}}} \exp (\alpha T_R (\omega_{2K-1})) dW_{y_{2K-1}} (\omega_{2K-1})\notag.
\end{align}

Now we claim that for every $\beta>0$ there exists $\alpha>0$   such that putting $q=1-p$ we have 
   \begin{enumerate}[(i)]
 \item \label{e:1} $\forall x\in S_R$, $\int \exp (\alpha T_r) dW_x \leq e^\beta$
 \item  \label{e:2} $\forall y\in S_r$, $\int_{G_y^c} \exp (\alpha T_R) dW_y \leq q e^\beta$ and $\int_{G_y} \exp (\alpha T_R) dW_y \leq p e^\beta$.
 \end{enumerate}
From this and \eqref{eq:integral} we immediately get that 
\begin{equation}\label{eq:exponential estimates} \int_{\Omega_x^K} \exp (\alpha T) dW_x \leq e^{2K\beta} q ^{K-1} p. \end{equation}
Thus, choosing $\beta$ such that  $e^{ 2\beta} q <1$ and summing over $K$, we obtain that 
 \[ \int \exp (\alpha T) dW_x \leq \frac{p}{q} \sum_K (e^{2\beta} q ) ^K <+\infty ,\]
 and the result follows.
 
\medskip

It remains to prove our claim. We only deal with (\ref{e:2}), the other case being analogous.
Let $0<\alpha < \alpha_0/2$, where $\alpha_0$ is in Corollary \ref{cor:franchi},
let  $\beta ' = \beta/2$ and set $E= G_y$ or $E= G_y^c$. Applying the Chebychev inequality, Corollary \ref{cor:franchi} implies that for every $y \in \partial _r$, we have 
\[ W_y ( T_R  \geq t) \leq M \exp (-\alpha_0 t). \]
We partition the set $E$ as 
$$E= \lrpar{E \cap \{ \alpha T_R \leq \beta'\} }\cup \bigcup_{n=0}^\infty
\lrpar{E \cap \{ \beta' + \alpha n < T_R < \beta' + \alpha (n+1) \} }$$
and decompose the   integral $\int _E \exp (\alpha T_R ) dW_y $ 
  accordingly, which  yields     
\begin{align*}
 \int _E \exp (\alpha T_R ) dW_y &\leq e^\beta W_y(E)  +  \sum _n \pp_x \left(T_R \geq \frac{\beta'}{\alpha} +n\right) e^{\beta' + \alpha (n+1)} \\& \leq e^{\beta'} W_x(E) + C \exp \left( -\frac{\alpha_0 \beta' }{ \alpha} \right)
 \end{align*}
with $C= \frac{e^{\alpha_0 + \beta'} M}{1-e^{-\alpha_0/2}}$. Note that $W_x(E) $ equals $p$ or $q$, that are both fixed positive numbers, so that if $\alpha$ is sufficiently small we have  that 
\[ \int _E \exp (\alpha T_R ) dW_y \leq e^{\beta'} W_x(E) + C \exp \left( -\frac{\alpha_0 \beta' }{ \alpha} \right) \leq e^\beta W_x (E) .\]
and claim~\eqref{e:2} follows. The proof is complete. 
\end{proof}

We can now show that the discretization measure satisfies the moment condition (M). 

\begin{prop}\label{prop:moment} Let $X$ be a hyperbolic Riemann surface of finite type, 
  $\rho$ be a parabolic representation of $G =  \pi_1(X)$ and $\mu$ be the discretization 
of the Brownian motion. Then there is a constant $\alpha = \alpha_\rho >0$ such that 
\begin{equation}\label{eq:moment}
  \int \norm{\rho (g)}^{\alpha} d\mu (g) <\infty.
\end{equation} 
Furthermore, the constant $\alpha_\rho$  can be chosen to be locally uniform in  
$\mathrm{Hom}_{\rm par} (\Gamma,\mathrm{PSL}(2,\mathbb C))$. In this   case the value of the integral  in \eqref{eq:moment} is also locally uniform.
\end{prop}

\begin{proof} 
For $\rho = \rho_{\rm can}$, given Proposition \ref{prop:exponential moment stopping time}, the proof 
is identical to that of Proposition \ref{prop:cocompact}. However in this case we cannot transfer the 
moment condition in $\Gamma$ to a moment condition with respect to word length in $G$ (see   
Remark \ref{rmk:moment} below). Instead we use the Lipschitz property of the spherical metrics on 
the fibers of $M_\rho$ to obtain moment estimates for $\rho(G)$. 

Let us be more precise. The proof of Proposition \ref{prop:cocompact} shows that the exponential 
decay property \eqref{eq:exp decay} implies that 
$$ \mu (\gamma\in \Gamma, \   d(\gamma(0), 0) \geq \delta) \leq e^{-c\delta}$$ for some positive 
constant $c$. 
Since $\norm{\gamma} \asymp \log d(\gamma 0, 0)$ we infer that there exists $\alpha>0$ such that 
$$\int_G \norm{\rho_{\rm can} (g)}^{\alpha} d\mu (g)<\infty.$$ Applying Corollary \ref{cor:compar}   thus finishes the proof.
 \end{proof}

\begin{rmk}\label{rmk:moment}  
 When $X$ is not compact, the moment condition in \eqref{eq:moment} cannot be replaced by \eqref{eq:exponential moment in the group}. Actually, 
 $\mu$ does not even have a finite moment with respect to word length in $G$, i.e. $\int_G \length(g) d\mu(g)=\infty$. This was shown (somehow implicitly) by Guivarc'h and Le Jan in \cite{guivarch lejan}. 
See also \cite[Corollary 1.22]{dkn} for a different approach, as well as Gruet \cite[Rmk p. 500]{gruet} for a neat computation in the case of the modular surface.
\end{rmk}

\subsection{Proof of Theorem \ref{thm:support brownien}}\label{subs:proof proof}
Let $\mu$ be the discretization measure and $\chi_\mu$ be the associated Lyapunov exponent. We showed in Proposition \ref{prop:moment} that $\mu$ satisfies the moment condition (M).
Furthermore, it is easy to see that  $\supp(\mu) = G$. Indeed, recall the  balayage construction of 
$\mu$. We start with the circle measure $ds$ on $S_R$ and ``sub-balayage" it successively  
on $\fr_R = \Gamma \cdot S_R$. The first step of the process is to balayage $ds$ onto $\fr_r =\Gamma \cdot S_r$. It is clear that the image measure has positive  mass on every circle 
$\gamma S_r$, $\gamma \in \Gamma$. Now for each circle  $\gamma S_r$, there is a set of positive 
measure   of good paths, that leaves  some measure on $\gamma S_R$ (for the remaining paths, we 
repeat the process over again). Therefore $\mu$ puts some mass on $\set{\gamma}$.

By Theorem \ref{thm:support}, $\supp(dd^c\chi_\mu) = \bif$. To prove the theorem, we show that 
there exists  a positive constant $c$ such that $\chi_{\rm Brown} = c\chi_\mu$, hence $\supp(\tbif) = 
\supp(T_{{\rm bif}, \mu}) =\bif$. Under (R3), the positivity and H\"older continuity properties of
 $\chi_{\rm Brown}$ then follow from classical results due  to
 Furstenberg \cite{furstenberg} and Le Page \cite{lepage holder} .

 Let $T$ be the stopping time defining $\mu$, and $T_n$ be the sequence of times obtained by repeating the stopping  process. 
Let $W_{S_R}$ be defined by $W_{S_R} =\int_{S_R} W_x ds(x)$. For $W_{S_R}$ a.e. $\omega$, 
$\omega(T_n) \in r_n(\omega)S_R$, with $(r_n)\in \Gamma^\nn$ as in \S \ref{sss:stopping time}. 

Identify $\Gamma$ and $G$ through $\rho_{\rm can}$ and consider  
a general parabolic representation $\rho$. Let as before $\sigma_{r_n}$ denote the geodesic segment  joining 
$0$ and $r_n( 0)$.
Since 
$d(r_n( 0), \omega(T_n))=R$ is constant we infer that  
$$\unsur{n}\abs{\log\norm{h_\rho({\sigma_{r_n(\omega)}}) }-  \log\norm{h_\rho({\omega(T_n)})}} 
= \unsur{n} \abs{\log\norm{\rho(r_n)} -  \log\norm{h_\rho({\omega(T_n)})}} \longrightarrow 0.$$ 
Since the law of $r_n$ is $\mu^n$ we thus deduce that 
\begin{equation}\label{eq:chibrown}
\chi_\mu(\rho) = 
\lim_{n\cv\infty} \unsur{n}\int_G \log\norm{\rho(g)}d\mu^n(g) 
= \lim_{n\cv\infty} \unsur{n} \int   \log\norm{h_\rho({\omega(T_n)})} dW_{S_R}(\omega).
\end{equation}
To conclude the proof, it remains to show that there exists a constant $\tau>0$ such that $\frac{T_n}{n}\cv  \tau $ for $W_{S_R}$ a.e. $\omega$. This, together with Proposition 
\ref{prop:lyap2}, implies that the right hand 
side of \eqref{eq:chibrown} equals $ \tau \chi_{\rm Brown}$

%

\medskip

To prove this, we view ${S_R}$ as a subset of $X$. Let $\theta_T$ be the time shift by $T(\omega)$, acting on the set of Brownian paths $\omega$ 
issued from $S_R$,  equipped with the measure $W_{S_R}$. 
By construction the law of $\omega(T)$ is $ds(x)$. From this fact and the strong Markov property
we see that $(\theta_T)_* W_{S_R} = W_{S_R}$.   
Furthermore there cannot be any measurable $\theta_T$-invariant subset, again by the strong Markov property 
together with the fact that    for every 
$x\in S_R$, the law of $\omega(T)$, conditioned to start at $x$, is $ds(x)$. Therefore $\theta_T$ is 
$W_{S_R}$-ergodic 

To conclude,  observe   that $T_n(\omega)  = \sum_{k=0}^{n-1} T(\theta_T^k \omega)$ and recall  
from Proposition \ref{prop:exponential moment stopping time} that $T$ is integrable.
Thus, property (ii) follows from Birkhoff's ergodic  theorem. The proof is complete.
\qed

\subsection{The Lyapunov exponent associated to  the  geodesic flow}\label{subs:geodesic}
 As before  let $X$ be a Riemann surface of finite type with a marked point $\star$, 
   endowed with its  hyperbolic metric. For  $v\in T^1_\star X$ a unit tangent vector, we let $\gamma_{v}$ be the 
   associated unit speed   geodesic ray. 
    These geodesics form   another natural class of paths on $X$.
    We also denote by $\widetilde\gamma_{v,t}$ the loop obtained by appending to 
   $\gamma_v\rest{[0,t]}$ a shortest path returning to $\star$ (as before, we disregard the  possible ambiguities happening here).
   This point of view  allows to define a Lyapunov exponent function, 
   similarly to Definition  \ref{def:lyap}. It turns out that it coincides with $\chi_{\rm Brown}$ up to a multiplicative constant. 
   Proposition \ref{propo:A} immediately follows.

\begin{thm}\label{thm:chi geodesic}
 Let $X$ be a hyperbolic Riemann surface of finite type, endowed with its Poincar\'e metric. Let $\rho:\pi_1(X, \star)\cv \PSL$ be a parabolic representation,  and $\norm{.}$ a smooth  Lipschitz  family of spherical metrics on the suspension $M_{\rho}$. Then for a.e. $v$ the limit 
$$\chi_{\rm geodesic}(\rho):=\lim_{t\cv\infty} \unsur{t}\log\norm{ h_\rho\big({\gamma_{v}\rest{[0,t]}}\big)} =
\lim_{t\cv\infty} \unsur{t}\log\norm{ h_\rho\big(   \widetilde\gamma_{v,t}  \big)}  $$ exists and
is equal to $\chi_{\rm Brown}(\rho)$ (in particular it does not depend on $v$).
\end{thm}

We leave the reader   check that the same result holds for   geodesics rays
issued from generic $(x,v)\in T^1X$ (relative to the Liouville measure), that is, 
$\lim_{t\cv\infty} \unsur{t}\log\norm{ h_\rho\big({\gamma_{(x,v)}\rest{[0,t]}}\big)} = \chi_{\rm Brown}(\rho)$. 

\begin{proof}
The work was already done in Proposition \ref{prop:lyap2}.
 It follows from \eqref{eq:lipschitz} that if $\omega_1$ and $\omega_2$ are two paths on $X$ with $\omega_1(0) = \omega_2(0)$ and  such that the respective lifts $\widetilde\omega_1$ and $\widetilde\omega_2$ to $\hh^2$, 
starting from the same point, satisfy $d_\hh(\widetilde\omega_1(t), \widetilde\omega_2(t))=o(t)$, 
then  $$\lim_{t\cv\infty} \unsur{t} \abs{  \log\norm{h_\rho({\omega_1}(t))} - \log\norm{h_\rho({\omega_2}(t))}}=0.$$ As already said, 
generic  Brownian paths on $\hh^2$ are shadowed by    geodesic rays. That is, if $\omega$ is a Brownian path 
starting from $0\in \hh^2$, then $\lim_{t\cv\infty} \omega(t) = \omega_\infty$ exists in $\fr\hh^2$, and
$d_\hh(\omega(t), \gamma_{0,\omega_\infty}( t ) )= o(t)$, where 
 if $\gamma_{0,\omega_\infty}$ is the unit speed   geodesic ray joining $0$ and $\omega_\infty$.
 
To prove the theorem, we use this reasoning in the reverse direction. 
 Taking endpoints of semi-infinite geodesics gives 
 the natural ``visibility'' identification between $T^1_0X$ and $\fr\hh^2$, sending circle measure on $T^1_0X$ to the 
 harmonic measure $\nu_0$ 
 associated to $0$ on $\fr \hh^2$. For $v\in T^1_\star X$, let $\omega_\infty(v)$ be the associated endpoint. For a.e. Brownian path $\omega$ conditioned to converge to $\omega_\infty(v)$ we have that 
$$\lim_{t\cv\infty} \unsur{t} \abs{  \log\norm{h_\rho ({\gamma_{v}} (t)) } - \log\norm{h_\rho({\omega }(t) )}}=0.$$ Now by Proposition \ref{prop:lyap}, for $W_\star$ a.e. $\omega$ we have that 
$\lim \unsur{t}\log\norm{h_\rho({\omega }(t))}  = \chi_{\rm Brown}$. Hence  
 $\lim \unsur{t}   \log\norm{h_\rho({\gamma_{v}}(t))} = \chi_{\rm Brown}(\rho)$. To deal with the loops $ \widetilde\gamma_{v,t}$, as in Proposition \ref{prop:lyap2} we use Sullivan's result that almost surely, $d_X(\gamma_v(t), \star)  =O(\log t)$. The proof is complete. 
\end{proof}

\begin{rmk}
In \cite{bonatti gomez mont viana}, Bonatti, Gomez-Mont and Viana study $\mathrm{PSL}(n,\cc)$ cocycles over hyperbolic and symbolic dynamical systems, with applications to suspensions of fundamental groups of compact surfaces. In particular they show, 
 that if $\rho: \pi_1(S)\cv \PSL$ is a non-elementary representation, then the foliated geodesic flow of 
 $M_\rho$ admits non-zero Lyapunov exponents (i.e. $\chi_{\rm geodesic}(\rho)>0$).  We see that the 
 above methods (discretization of Brownian motion, etc.) give a new approach to this result, as it was    
 recently observed by Alvarez \cite{alvarez cras} (for compact $S$).  
The novelty here is that  we are able to deal with parabolic representations on Riemann surfaces of finite type as well. 
\end{rmk}

\section{Equidistribution in parameter space}\label{sec:equidist}

Let as before $X$ be a structure of  Riemann surface of finite type on $S$. Identifying $(\widetilde X, \star)$ with $(\hh, 0)$
 makes $G = \pi_1(S, \star)$ 
isomorphic to a lattice $\Gamma\leq \mathrm{PSL}(2,\rr)$.  From now on, we  consistently  identify $G$ and $\Gamma$. We endow $\Gamma$ with the metric structure induced by $\hh$ on $\Gamma\cdot 0$, that is 
for  $\gamma\in \Gamma$ we put $d(\gamma)  = d_\hh(0, \gamma 0)$; observe that 
$d(\gamma)\sim 2 \log \norm{\gamma}$.
 Notice that when $X$ is not compact, this distance is {\em not} quasi-isometric to  word-length in $G$. We also put 
 $B_\Gamma( r) = \set{\gamma\in \Gamma, \ d(\gamma)<r)}$.
  
  \medskip
  
This metric structure 
gives rise to a natural notion of a random diverging sequence in $\Gamma$: fix a sequence $(r_n)_{n\geq 1}$ 
increasing to infinity and 
independently choose  $\gamma_n \in B_\Gamma ( r_n)$ relative to the counting measure. For technical reasons we will 
assume that the series $\sum_{n\geq 0} e^{-cr_n}$ converges for every $c>0$. The resulting  random sequences $(\gamma_n)$ will be qualified as {\em admissible}.  
  
   Let now $\La$ be a holomorphic family of parabolic representations  of $G$ into $\PSL$ (resp.  a holomorphic family of representations modulo conjugacy)  satisfying (R1, R2, R3).
For $t\in \cc$, we consider the   subvariety in $\La$ defined by   
$$Z(\gamma, t) = \set{\la\in \La, \ \tr^2( \rho_\la({\gamma})) = t}$$ (which is viewed as a variety, that is, possibly with multiplicity), and the associated integration current $[ Z(\gamma, t)]$. By convention if $Z(\gamma, t)  = \La$, we put  
$[ Z(\gamma, t)] =0$. 

Our purpose in this section is to establish the following theorem. 
  
\begin{thm}\label{thm:equidist}
Let $\Gamma$ be a torsion-free lattice in $\mathrm{PSL}(2, \rr)$, and 
 $(\rho_{\la})_{\la\in \La}$ be a holomorphic family of parabolic 
representations of $\Gamma$ into $\PSL$ satisfying (R1, R2, R3). Let  
$T_{\bif} = dd^c\chi_{\rm Brown}$ be the  natural bifurcation current, associated to Brownian motion on $X = \Gamma\setminus \hh$. 
If  $(\gamma_n)$ is an admissible  random sequence  in $\Gamma$, then almost surely the following convergence holds:
$$\unsur{2d(\gamma_n)} \left[ Z(\gamma_n, t)\right] \underset{n\cv\infty}\longrightarrow   \tbif.$$ 
 
\end{thm}

As already said in the introduction,
    this result  cannot be deduced from the equidistribution results in \cite{kleinbif}.  
 The main step of the proof, achieved in \S  \ref{subs:transport},
 is to show  that if $\rho\in   \mathrm{Hom_{par}}(\Gamma, \PSL)$ is a {\em fixed} representation, 
 then the  exponential growth rate of  $\tr^2 ( \rho({\gamma_n}))$ is  a.s. given by 
 $\chi_{\rm Brown}(\rho)$. 
 This is based on  the detailed understanding of the  asymptotic  properties of random product of matrices 
 associated to discretized Brownian motion. 
 
 Then in \S\ref{subs:subharmonic} we  use Fubini's theorem and 
  subharmonicity  to get the  almost sure 
 $L^1_{\rm loc}$ convergence of  
 $\la\mapsto   \log\abs{\tr^2( \rho_\la({\gamma}_n)) - t}$. 
 From this,  Theorem \ref{thm:equidist} immediately follows.
This simple argument was suggested to us by  D. Chafa\"i. Notice that the same approach leads to a simplification of the almost sure 
 equidistribution theorems in 
\cite{kleinbif} as well (see Theorems B and 4.1 there). 
 
 In \S\ref{subs:models} we show that this result can be nicely interpreted in terms of the holonomies of 
 random sequences of closed geodesics on $X$ (Theorems \ref{theo:counting} and \ref{theo:equidist}).

\subsection{Distribution of traces in $\rho(\Gamma)$}\label{subs:transport} 
Let $\Gamma$ be a lattice in  $\mathrm{PSL}(2, \rr)$. If $\Gamma$ is torsion-free and $\rho$ is a non-elementary 
type-preserving representation of 
$\Gamma$ into $\PSL$,  the results of   \S\ref{sec:brownian} 
enable to    define the natural Lyapunov exponent $\chi_{\rm Brown}(\rho)$. If now 
$\Gamma$ has torsion, by the Selberg Lemma, we   pick a finite index torsion-free subgroup $\Gamma'\leq \Gamma$, and  
define $\chi_{\rm Brown}(\rho)$  to be equal to $\chi_{\rm Brown}(\rho\rest{\Gamma'})$. By Proposition \ref{prop:covering},  the value of the Lyapunov exponent does not depend on the choice of $\Gamma'$. 
 
Our main result in this paragraph is the following.

\begin{thm}\label{thm:counting}
Let $\Gamma$ be a   lattice in $\mathrm{PSL}(2, \rr)$, and   $\rho$ be a non-elementary 
 parabolic  representation of $\Gamma$.  
Then for every $\e>0$ there exists $c(\e)>0$ such that for large enough $r>0$,
$$\# \set{\gamma\in   B_\Gamma( r), \ \abs{ 
\unsur{{2r}}\log \abs{\tr^2 (\rho(\gamma))}   - \chi_{\rm Brown}(\rho) }>\e}
\leq e^{(1-c(\e))r}.$$
\end{thm}

The  delicate point in this theorem is the following: several times so far we have used the idea that closing a long geodesic or 
Brownian path  by some path of bounded or short length does not affect much the norm of the holonomy. There is no such result for the 
trace: if $\gamma$ (resp. $a$)  is a ``large" (resp. ``small") 
element in $\Gamma$,  $\tr^2(\rho(\gamma))$ may be very different from  
$\tr^2(\rho(\gamma a))$ (even for the identity representation). 

The idea of the proof is to rely on the very precise understanding of  generic  random products 
$\gamma_1\cdots \gamma_n$, where the $\gamma_i$ 
are chosen according  to the discretization measure $\mu$ defined in \S\ref{subs:discretization}. 
In Proposition \ref{prop:trace perturbee} we show that  
we can  perturb these elements in $\rho(\Gamma)$ by keeping track of  the traces. We then 
use the properties of Brownian motion to compare
 $\mu^n$ and the counting measure, so as to cover most of a large ball in $\Gamma$.
 

\medskip
 
Throughout the proof   $c(\e)$ stands 
for a quantity which may change from line to line, dependent on $\e$ (and $\rho$) but not on $n$.


Our first result does not depend on Brownian motion. 

\begin{prop}\label{prop:trace perturbee}
Let $\Gamma$ be as above and $\rho: \Gamma\cv\PSL$ be a non-elementary representation. 
Fix a probability measure $\mu$ on $\Gamma$, generating $\Gamma$ 
as a semi-group, and satisfying the exponential moment condition 
$\int_\Gamma \norm{\rho(\gamma)}^\alpha d\mu(\gamma)<\infty$ for some $\alpha>0$.  

Then  there exists a constant $c=c(\mu, \rho)>0$ such that for every $\e>0$, 
 when $n$ is large enough
\begin{equation}\label{eq:trace perturbee}
 \mu^n\lrpar{\set{\gamma, \ \exists a\in \Gamma \text{ s.t.  } \norm{a}\leq e^{ c  \e n}  \text{ and  } \abs{
\unsur{n} \log \abs{\tr \rho(\gamma a)} - \chi_\mu}>\e }} \leq \exp(-c\e n).
\end{equation}
\end{prop}

We note that in this statement 
 we could replace $\Gamma$ by an abstract finitely generated group,  by replacing the condition 
``$\norm{a}\leq e^{ c  \e n}  $" in \eqref{eq:trace perturbee} by ``$\length(a)\leq c\e n$". 

\begin{proof} 
Without the perturbative term $a$, this was proven in \cite[Cor. A.2]{kleinbif} (see also \cite{aoun}). 
It turns out that for fixed $a$, the proof is essentially identical to that of  Theorem A.1 in \cite{kleinbif}, 
which we briefly reproduce for convenience. For a sequence 
$\boldsymbol{\gamma} \in \Gamma^\nn $, we put 
$\ell_n(\boldsymbol{\gamma}) = \gamma_1\cdots \gamma_n$. We recall   that if $g\in \PSL$ has large norm, 
 there exists two balls $A(g)$ and $R(g)$ 
 (the attracting and repelling balls respectively) on $\pu$, of diameter about $\unsur{\norm{g}}$, 
such that $g(R(g)^c) = A(g)$. We also recall that for $g\in \PSL$,
$$\unsur{2}\log {\abs{\tr^2(g) -4}} = \log  \norm{g} + \log \delta(g) + O(1),$$ 
where $\delta(g)$ is the distance between the fixed points of $g$. Thus to control the trace from the norm it is enough to control the distance between fixed points. 

Fix $\e>0$, small with respect to $\chi$.  
Fix $a\in \Gamma$, with $\norm{a}\leq  \frac{\e}{2}n$.
 We have to estimate the probability that for large $n$
$\unsur{n} \log \tr^2(\rho (\ell_n(\boldsymbol{\gamma})a))\notin[\chi_\mu-\e, \chi_\mu+\e]$. 
By the large deviations theorem for the distribution of the norms \cite[\S V.6]{bougerol lacroix}
 we infer that outside a set of 
exponentially small probability, 
$\abs{\unsur{n} \log \norm{\rho (\ell_n(\boldsymbol{\gamma}))} -\chi_\mu} <\frac{\e}{4}$. Since 
$\norm{a}\leq e^{\frac{\e}{2}n}$, we get that with high probability, 
 $$\abs{\unsur{n} \log \norm{\rho (\ell_n(\boldsymbol{\gamma})a)} -\chi_\mu} <\frac{3\e}{4}.$$ Therefore we need to show that the 
 probability that $\delta(\rho (\ell_n(\boldsymbol{\gamma})a)) < e^{-\frac{\e}{4} n}$ decreases like $e^{-c(\e)n}$. 
 
 \medskip

For this,  let $m=\lfloor (1-\e)n \rfloor$. 
We split $\ell_n(\boldsymbol{\gamma})a$ as $\ell_n(\boldsymbol{\gamma})a = 
(\gamma_n\cdots \gamma_{m+1}) (
\gamma_{m}\cdots \gamma_1 a)$. 
By the  large deviations theorem for the norms, and the facts that 
$\norm{a}\leq e^{\frac{\e}{2}n}$ and $m=\lfloor (1-\e)n \rfloor$,  we  
get that 
$$ \mathbb P \Big( \Big| \frac{1}{m} \log \norm{ \rho (\ell_m(\boldsymbol{\gamma})a )} -\chi_\mu \Big| \geq  \varepsilon  \Big)+\mathbb P \Big( \log \norm{\rho(\gamma_n \ldots \gamma_{m+1}) } \geq  2 \chi_\mu (n-m) \Big) \leq e^{-c(\e)n}.
$$
 By independence, we may assume that $\gamma_1 , \ldots, \gamma_m$ are fixed and satisfy
 $\abs{ \frac{1}{m} \log \norm{ \rho (\ell_m(\boldsymbol{\gamma})a )} -\chi_\mu }< \varepsilon$
 and estimate the  probability that 
 $\delta(\rho (\ell_n(\boldsymbol{\gamma})a)) < e^{-\frac{\e}{4} n}$ 
among the  $\gamma_{m+1},\ldots , \gamma_n$ satisfying 
$\log \norm{\rho(\gamma_n \ldots \gamma_{m+1}) } <  2 \chi_\mu (n-m) \sim 2\chi n\e $. 
Put $h = \rho  (\ell_m(\boldsymbol{\gamma})a )$ and consider the balls $A(h)$ and $R(h)$. If we let 
$$\widetilde{A}( \rho (\ell_n(\boldsymbol{\gamma})a ) = \rho (\gamma_n\cdots \gamma_{m+1}) A(h) \text{ and } 
\widetilde{R}( \rho (\ell_n(\boldsymbol{\gamma})a )=R(h)$$ then by definition we have that 
$$\rho (\ell_n(\boldsymbol{\gamma})a)\left (\widetilde{R}( \rho (\ell_n(\boldsymbol{\gamma})a )^c\right) = 
\widetilde{A}( \rho (\ell_n(\boldsymbol{\gamma})a ) 
$$ and the diameters of 
${\widetilde{A}( \rho (\ell_n(\boldsymbol{\gamma})a )} $  and 
${\widetilde{R}( \rho (\ell_n(\boldsymbol{\gamma})a )}$ are   not greater than $e^{-(\chi_\mu + O(\e))n}$.
So if we can show that these balls are separated by a distance $\geq e^{-\frac{\e}{4} n}$,   we then conclude that 
 $\rho (\ell_n(\boldsymbol{\gamma})a )$ is loxodromic with a fixed point in each ball, so the distance between its fixed points is 
 $\geq e^{-\frac{\e}{4} n}$ and we are done. But if this does not happen, this means that $\rho(\gamma_n\cdots \gamma_{m+1})$ maps the ball $A(h)$ very close (i.e. at distance $\leq e^{-\frac{\e}{4} n}$)  to the ball $R(h)$. Since   $h$ is fixed here, the probability of this event is governed by the properties of the unique stationary   measure  on $\pu$, relative 
 to the action of $(\rho(\Gamma), \mu)$. It follows that this probability is of order $e^{-\theta \e n}$ for some $\theta= \theta(\rho)>0$ (see \cite[Lemma A.3]{kleinbif} for details) and the desired result follows. 
 
\medskip

So far we have shown that there exists $c_1>0$ (which we may assume is smaller than $\frac14$) such that 
 if $a$ is a fixed element in $\Gamma$ such that  $\norm{a}\leq e^{\frac{\e}{4}n}$, then 
$$\pp\lrpar{ \unsur{n} \log \abs{\tr \rho(\ell_n(\boldsymbol{\gamma})a)}\notin[\chi_\mu-\e, \chi_\mu+\e] }\leq e^{-c_1\e n}.$$ Since the number of elements of norm not greater  than  $e^{cn}$  is  equivalent to $e^{cn}$, we get that for $c_2\leq c_1$, 
$$\pp\lrpar{  \exists a\in \Gamma , \text{ s.t.  }    \norm{a} \leq e^{c_2 \e n}, \text{ and }
\abs{ \unsur{n} \log \abs{\tr \rho(\ell_n(\boldsymbol{\gamma})a)} - \chi_\mu}>\e } \leq 
e^{(c_2-c_1)n}. $$ By choosing $c_2   =  \frac{c_1}{2}$,
we conclude that \eqref{eq:trace perturbee} holds with   $c= \frac{c_1}{2}$. 
 \end{proof}

\begin{proof}[Proof of Theorem \ref{thm:counting}] Let us first assume that $\Gamma$ is torsion-free.
From now on  $\mu$ is the discretization measure for the  Brownian motion on $X$, as defined in \S \ref{subs:discretization}. 
We let  $S=S_R$ be the circle centered at $0$ and of radius $R$ and  work in the    probability space of paths issued from $S$ in $\hh$, endowed with the probability measure
 $W_{S} = \int_S W_xds(x)$. 
For a  generic such 
 Brownian path,   we
let $T = T(\omega)$ be the stopping time defining $\mu$, $T_n = T_n(\omega)$ be the sequence of times obtained by repeating the 
stopping process, and $\tau$ be the average stopping time, which is the almost sure limit $\tau =\lim \frac{T_n}{n}$ (see the proof of Theorem \ref{thm:support brownien}). Observe that 
it is enough to prove the theorem for the sequence of radii      $ r_n = {n\tau} $.  

\medskip

The first step is the following large deviations estimate, which will be proven afterwards. 

\begin{lem}\label{lem:large dev stopping time} 
Let $T_n(\omega)$ be as above. Then 
for large enough $n$ we have that 
$$W_{S} \lrpar{\set{\omega, \  \abs{\frac{T_n(\omega)}{n} - \tau}>\e} }\leq e^{-c(\e)n}.$$
\end{lem}

For fixed $n$, we now define a good set $G_n$  of paths starting from $S$ by specifying its complement $B_n$. 
Declare that the elements   $\gamma\in \Gamma$ appearing in \eqref{eq:trace perturbee} are bad. 
From the discretization procedure, this corresponds to a set  of  
 paths of exponentially small measure, relative to  $W_{S}$, which we denote by 
    $B_n^0$. 
Now by definition $\omega$ belongs to $B_n^1$  if   
$\abs{T_n(\omega) - n\tau}> \frac{c\tau}{2K}\e$ where $c$ is as in Proposition 
\ref{prop:trace perturbee} and $K$ is as in  Lemma \ref{lem:martingale}. 
By Lemma \ref{lem:martingale} and the strong Markov property  if $\omega\notin B_n^1$, the probability that
$d_\hh (\omega(T_n), \omega(n\tau))$ is larger than $\frac{c \tau}{2} n\e$ is exponentially small. 
We let $B_n^2$ be the corresponding  set of paths. 
Finally we put $B_n = B_n^0\cup B_n^1\cup B_n^2$, and $G_n = B_n^c$. Observe that $W_{S}(B_n)\leq e^{-c(\e)n}$.

We claim that if $\omega$ is good, that is $\omega\in G_n$, then for every $\gamma\in \Gamma\cap B_\hh (\omega(\tau n), \frac{c\tau}{2}n\e)$ we have that 
$$\abs{\unsur{n}\log \abs{\tr (\rho(\gamma))} - \chi_\mu(\rho)}<\tau \e 
\text{, i.e. }  
 \abs{\unsur{2\tau n}\log \abs{\tr^2 (\rho(\gamma))} - \chi_{\rm Brown} (\rho)}<\e$$ (recall from the proof of Theorem \ref{thm:support 
 brownien} that $\chi_\mu (\rho)= \tau \chi_{\rm Brown}(\rho)$). Indeed by Proposition \ref{prop:trace perturbee}, 
 if $\omega\notin B_n$, then the element $\alpha\in \Gamma$ 
 corresponding to $\omega(T_n)$  (that is,  such that $\omega(T_n)\in \alpha(S)$)  has the property that for every 
 $\beta \in B_\Gamma(\alpha, c\tau n\e)$, 
$\abs{\unsur{n}\log \abs{\tr (\rho(\beta))} - \chi_\mu(\rho)}<\tau \e$. Now, if $\omega\in G_n$, 
$\omega(\tau n)$ lies in  $B_\Gamma(\alpha, \frac{c\tau }{2}n\e)$, whence the result. 

\medskip

Let $E$ be the set of group elements $\gamma \in B_\Gamma\lrpar{\frac{n\tau}{2}}$ such that 
$$\abs{\frac{1}{2\tau n} \log\abs{\tr^2(\rho(\gamma))} -\chi_{\rm Brown}(\rho)}\geq \e.$$ To complete the proof of the theorem, we 
have to show that   $\# E   
\leq e^{-c(\e)n} \#B_\Gamma\lrpar{{n\tau}}$. For this, let $\widetilde G_n$ be the image of $G_n$ under Brownian motion 
at time $n\tau$, that is $\widetilde G_n = \set{\omega(n\tau), \ \omega\in G_n}$. 
The next lemma follows from classical estimates for  the heat kernel on $\hh$ (see below for the proof). 

\begin{lem}\label{lem:heat}
The hyperbolic area of $(\widetilde G_n)^c \cap B_\hh\lrpar{0, {n\tau}}$ is exponentially small with respect to the area of $B_\hh \lrpar{0, {n\tau}}$, that is, there exists $c(\e)>0$ such that for large $n$, 
$$\mathrm{Area} \left((\widetilde G_n)^c \cap B_\hh\lrpar{0, {n\tau}} \right)\leq \exp\lrpar{{n\tau} - c(\e) n}.$$
 \end{lem}

We have shown above  that if $\gamma\in E$, then $B_\hh(\gamma,  \frac{c\tau\e}{2}n)$ is disjoint from $\widetilde{G_n}$. 
Let $\delta = \frac{c(\e)}{2}$, where $c(\e)$ is as in the previous lemma.   By the 
Vitali covering lemma, there exists a covering of  $E$ by a finite collection  of balls $\set{B_\hh(\gamma_i, \delta n), i=1, 
\ldots N}$ such that the balls $B_\hh(\gamma_i,  \frac{\delta}{5}n )$ are disjoint. Thus we infer that 
\begin{equation}\label{eq:N}
\sum_{i=1}^N \mathrm{Area} \lrpar{B_\hh\lrpar{\gamma_i,  \frac{\delta}{5}n} }
= 
4\pi N \sinh^2 \lrpar{ \frac{\delta}{10}n }
 \leq  \mathrm{Area}  \lrpar{(\widetilde G_n)^c}.
\end{equation} Now, a well known theorem due to Margulis \cite{margulis these}  (see also \cite{eskin mcmullen})
 asserts that when $r\cv\infty$, 
\begin{equation}\label{eq:margulis}
\#B_\Gamma ( r) \sim 
  \frac{\mathrm{Area} (B_\hh(0, r))}{\mathrm{Area}(X)} \sim \frac{e^r}{\mathrm{Area}(X)}.
  \end{equation}   It follows that there is a   constant $C$ such that for every $r$, $\#B_\Gamma(r) \leq C e^r$, hence  by \eqref{eq:N} 
  we conclude that 
  $$\# E \lesssim N   e^{\delta n} \lesssim   \frac{e^{{n\tau} - c(\e) n}}{\sinh^2\lrpar{\frac{\delta}{10}n}}
   e^{\delta n}  \lesssim
    e^{ {\tau}n - \lrpar{ c(\e)   -\frac45 \delta} n}\leq C e^{{\tau}n - \frac{c(\e)}{2} n},$$  
and  the proof of the theorem is complete in the torsion free case. 
\end{proof}

\begin{proof}[Proof of Lemma \ref{lem:large dev stopping time}]  
The almost sure convergence $\frac{T_n}{n}\cv \tau$ was established in the proof of  
Theorem \ref{thm:support brownien}, the point here is to make this convergence more precise. As we did there we slightly shift our point of view and see our 
 Brownian paths as sitting  on $X$. 
 We want to formalize the idea of a Brownian path issued from $S$ as obtained from the concatenation of path segments 
 starting and ending on $S$. For this we let $\om_S$ be the set of continuous paths $[0, \infty)\cv X$ starting on $S$, 
 endowed as usual with the probability measure $W_S$. We define a Markov chain with values in $\om_S$ as follows: the initial 
 distribution is $W_S$, and for a.e. $\omega$ the transition kernel is given by 
 $p(\omega, \cdot) = W_{\omega(T(\omega))}$, 
 where $T(\omega)$ is our stopping time. This Markov chain yields a sequence of random paths $(\omega_n)$  
 in $\Omega_S$. By the strong Markov property of Brownian motion, the law of the random 
 path obtained by concatenating the path segments  $\omega_n\rest{[0, T(\omega_n)]}$ is $W_S$. 
 Associated to the Markov chain there is a real random variable $T(\omega_n)$ and we 
 see that the law of $T_n(\omega)$ is the same as that of $\sum_{k=1}^n T(\omega_k)$. 
 
It is easy to see that  the random variables $(\omega_n)$ have  the same law $W_S$. However they are not independent since 
$\omega_{n+1}$ must start at $\omega_n(T(\omega_n))$. The main observation 
 is that independence holds after two iterations. Indeed for 
every $x\in S$, 
$$p(W_x, \cdot) = \int p(\alpha, \cdot) W_x(d \alpha) = \int W_{\alpha(T(\alpha))}  W_x(d \alpha) =\int W_y ds(y) = W_S$$ because
by construction  the 
law of $\alpha(T(\alpha))$ is the circle measure $ds$. Thus we infer that for a.e. $\omega\in S$, 
$$p^2(\omega, \cdot) =  {\int p(\alpha, \cdot) p(\omega, d\alpha)} = \int p(\alpha, \cdot) W_{\omega(T(\omega))}(d\alpha) =W_S.$$  
It follows that the random variables $\omega_1, \omega_3, \omega_5, \ldots$ 
(resp. $\omega_2, \omega_4, \cdots$) are independent. Furthermore, 
by Proposition \ref{prop:exponential moment stopping time}, there exists $s>0$ such that 
$\ee(e^{sT(\omega)})<\infty$.
 Therefore  the desired estimate follows 
from the classical large deviation theorem for partial sums of i.i.d. random variables 
(applied separately to $\sum_{ \mathrm{even}\ k} T(\omega_k)$ and  $\sum_{ \mathrm{odd}\ k} T(\omega_k)$). 
 \end{proof}

\begin{proof}[Proof of Lemma \ref{lem:heat}]
Let $\nu_t$ (resp. $\nu_t^0$) be the law at time $t$ of Brownian motion with initial distribution $ds$   (resp. $\delta_0$). 
For notational ease, we put  $F = \widetilde {G}_n^c \cap B_\hh\lrpar{0, {n\tau}}$. 
We know that  $\nu_{n\tau}(F)\leq W_{S}(B_n)\leq e^{-c(\e)n}$, and 
from this we want to derive an estimate for the hyperbolic area of $F$. 
Let $\delta>0$ to be specified later, and decompose $F$ as 
$$F=\left(F\cap B_\hh\left(0, {n\tau} - n\delta\right)\right) \cup  \left(F\cap A_\hh\left({n\tau} - n\delta, {n\tau}\right)\right)$$ where $A_\hh(r,s)$ denotes the annular region centered at 0 with inner radius $r$ and outer radius $s$. 
The first observation is that  $$\mathrm{Area}\left(F\cap B_\hh\left(0, {n\tau} - n\delta\right)\right)
\leq \mathrm{Area} \left(B_\hh\left(0, {n\tau}- n\delta\right) \right) \lesssim 
\exp\lrpar{{n\tau} - n\delta}.$$

The probability measure $\nu_t^0$ is radial, and 
from the estimate for its 
 radial density $k^0(t,r)$ given in \eqref{eq:davies}, we deduce that 
for $t=n\tau$ and $ {n\tau} - n\delta\leq r\leq {n\tau}$, 
$$\unsur{\sqrt{n}} e^{-n\frac{\delta^2}{4\tau}}\lesssim k^0( n\tau ,r)\lesssim  \unsur{\sqrt{n}}. $$ Since $\nu_{n\tau}$ is an average of bounded translates of $\nu_{n\tau}^0$, the  same estimate holds for the radial density $k(n\tau, r)$ of $\nu_{n\tau}$. 
On the other hand the radial density of hyperbolic area is $r\sinh r$, so we infer that 
\begin{align*}
\mathrm{Area}\lrpar{F\cap A_\hh\lrpar{n\tau - n\delta, n\tau}} &\leq 
\int_{F\cap A_\hh\lrpar{n\tau - n\delta, n\tau}} r(\sinh r) dr d\theta\\
&\lesssim  \int_{F\cap A_\hh\lrpar{n\tau - n\delta, n\tau}} r(\sinh r) k( n\tau ,r) \sqrt{n} e^{n\frac{\delta^2}{4\tau}} drd\theta
\\
&\lesssim n^{3/2}\exp\lrpar{n\tau+  \frac{\delta^2}{4\tau}n} \nu_{n\tau}(F)\leq n^{3/2}
\exp\lrpar{n\tau+  \frac{\delta^2 }{4\tau}n -c(\e)n}. 
\end{align*}
Finally we get that 
$$\mathrm{Area}(F) \lesssim \exp\lrpar{n\tau - n\delta} + n^{3/2}
\exp\lrpar{ {n\tau}+  \frac{\delta^2}{4\tau}n -c(\e)n},$$ so by choosing $\delta =\sqrt{\tau c(\e)}$ we are done.
\end{proof}

\medskip

Let us now assume that $\Gamma$ has torsion, and briefly explain how to adapt the proof of the theorem. We pick a torsion-free 
subgroup $\Gamma'\leq \Gamma$ of finite index, and apply the first part of the proof to $\Gamma'$. In this way we get a 
discretization measure $\mu$ on $\Gamma'$. We slightly modify Proposition \ref{prop:trace perturbee} to 
allow for a perturbation $a$ belonging to $\Gamma$ 
(this only slightly affects the counting in the last part of the proof), so that its conclusion   now reads  
$$\mu^n\lrpar{\set{\gamma\in \Gamma', \ \exists a\in \Gamma \text{ s.t.  } \norm{a}\leq e^{ c  \e n}  \text{ and  } \abs{
\unsur{n} \log \abs{\tr \rho(\gamma a)} - \chi_\mu}>\e }} \leq \exp(-c\e n).$$ Then we argue exactly as in the torsion free case, by introducing the set $E$ of $\gamma \in \Gamma\cap B_\hh\lrpar{{n\tau}}$ such that 
 $\abs{\frac{1}{2\tau n} \log\abs{\tr^2(\rho(\gamma))} -\chi_{\rm Brown}(\rho)}\geq \e.$ To the discretized Brownian motion 
 on $\Gamma'$, we associate the sets $G_n$ and $\widetilde G_n$ exactly as before, 
 and we estimate $\#E$ by arguing that, from the modified 
 version of Proposition \ref{prop:trace perturbee}, if $\gamma\in E$, 
 then $B_\hh(\gamma,  \frac{c\tau\e}{2}n)$ is disjoint from $\widetilde{G_n}$. 
This completes the proof of the theorem. 
 \qed

The following corollary, of independent interest, is the key step in the proof of Theorem \ref{thm:equidist}. 
 
\begin{cor}\label{cor:fixed param}
Let $\Gamma$ be a  lattice in $\mathrm{PSL}(2, \rr)$, and   $\rho$ be a non-elementary 
 parabolic 
  representation of $\Gamma$. If $(\gamma_n)$ is an admissible random sequence in $\Gamma$, then almost surely,
 $$\unsur{2d(\gamma_n)} \log\abs{\tr^2(\rho(\gamma_n))} \underset{n\cv\infty}\longrightarrow \chi_{\rm Brown}(\rho).$$
\end{cor}

\begin{proof}
Recall that admissible means that the  $\gamma_n$ are chosen randomly in $B_\Gamma(r_n)$, relative to the counting measure, 
where $r_n$ is a sequence such that $\sum e^{-cr_n}$ converges for every $c>0$. 

Fix $\e>0$. We will show that 
\begin{equation}\label{eq:borel cantelli}
\unsur{\#  B_\Gamma(r_n) }\  \# \set{\gamma \in B_\Gamma(r_n), \ \abs{\unsur{2d(\gamma_n)} \log\abs{\tr^2(\rho(\gamma_n))} - \chi_{\rm Brown}(\rho) }> \e} \lesssim e^{-c(\e)r_n}.
\end{equation}
We then deduce from the Borel-Cantelli lemma that  a.s. for large $n$, 
$\unsur{2d(\gamma_n)} \log\abs{\tr^2(\rho(\gamma_n))}$ is  $\e$-close to $\chi_{\rm Brown}(\rho) < \e$, and the result follows.

\medskip

The first observation is that in $B_\Gamma( r_n)$,  by the Margulis estimate \eqref{eq:margulis} on $\#B_\Gamma(r)$, 
for every $\delta>0$, we have that 
\begin{equation}\label{eq:cardinal}
\#\set {\gamma, \ d(\gamma_n) < (1-\delta)r_n} \lesssim e^{(1-\delta) r_n}. 
\end{equation}
We choose $\delta =  \min \lrpar{\unsur{2}, \frac{\e}{4\chi}}$. 
Now if $$d(\gamma_n) \geq  (1-\delta)r_n \text{ and }
\abs{\unsur{2d(\gamma_n)} \log\abs{\tr^2(\rho(\gamma_n)) }- \chi_{\rm Brown}(\rho) }> \e,$$ 
an elementary computation shows that 
$$\abs{\unsur{2r_n} \log\abs{\tr^2(\rho(\gamma_n))} - \chi_{\rm Brown}(\rho) } > (1-\delta)\e-\delta \chi_{\rm Brown}(\rho) > \frac{\e}{4}.$$
Therefore the desired estimate \eqref{eq:borel cantelli} follows from \eqref{eq:cardinal}, 
   Theorem \ref{thm:counting} and  
the fact that 
$\#\  B_\Gamma(r_n)$ grows like  ${e^{r_n}}$.
\end{proof}

\begin{rmk}\label{rmk:lattice}
This corollary yields yet another 
definition of the natural Lyapunov exponent, where the choice of Riemann surface structure is reflected in 
the fact of counting in the lattice $\Gamma$.  Similarly, we could  replace the trace by the norm, and
 assert that generically, $\unsur{2d(\gamma_n)} \log {\norm{\rho(\gamma_n)}}$ converges to 
 $\chi_{\rm Brown}(\rho)$. The proof is   identical to that of Theorem \ref{thm:counting} 
 (observe that Proposition \ref{prop:trace perturbee} becomes much simpler in this case) . 
\end{rmk}

\subsection{Proof of Theorem \ref{thm:equidist}}\label{subs:subharmonic}
Let $m_n$ be the normalized counting measure on $B_\Gamma(r_n)$ and $m = \prod_{n\geq 1} m_n$. 
By considering the potentials of the currents $\unsur{2d(\gamma_n)} [Z(\gamma_n, t)]$, 
it is enough to show that for $m$-a.e.  sequence $\boldsymbol{\gamma} = (\gamma_n)$ 
 in $\Gamma^\nn$, the sequence of psh functions  
 $(\la\mapsto u (\la, \boldsymbol{\gamma}, n))_{n\geq 1}$ on $\La$ defined by 
 $$u(\la, \boldsymbol{\gamma}, n) =  \unsur{2d(\gamma_n)} \log \abs{\tr^2(\rho_\la(\gamma_n))-t}$$
  converges in $L^1_{\rm loc}(\La)$ to $\la\mapsto \chi_{\rm Brown}(\rho_\la)$. Corollary \ref{cor:fixed param} above implies that for 
  every {\em fixed} parameter $\la$,  $u(\la, \boldsymbol{\gamma}, n)$ converges $m$-almost surely to $\chi_{\rm Brown}(\rho_\la)$. 
 Now clearly the set $E\subset \La\times \Gamma^\nn$ of parameters $(\la, \boldsymbol{\gamma})$ such that 
  $u(\la, \boldsymbol{\gamma}, n)$ converges to $\chi_{\rm Brown}(\rho_\la)$ is measurable, so 
  by Fubini's theorem   it has full $ d\la \otimes m$ measure. Reverting the order of integration  (Fubini's theorem again), we infer that 
  for $m$ a.e. $\boldsymbol{\gamma}$, $u(\la, \boldsymbol{\gamma}, n)$ converges   to $\chi_{\rm Brown}(\rho_\la)$ $d\la$-a.s.
 
  In addition, the sequence $u(\la, \boldsymbol{\gamma}, n)$ is locally uniformly bounded from above on $\La$. Indeed 
   we have  that 
 \begin{align*}
 \unsur{2d(\gamma_n)} \log \abs{\tr^2(\rho_\la(\gamma_n))-t} &\leq \unsur{2d(\gamma_n)} 
 \log \norm{\rho_\la(\gamma_n))}^2 +O(1)\\ &\leq \frac{\beta_{\rho_\la}}{d(\gamma_n)} \log \norm{\gamma_n} + O(1) 
 \leq \frac{\beta_{\rho_\la}}{2} + O(1)
 \end{align*}
  where the constant
 $\beta_{\rho_\la}$  
 is locally uniform in $\La$ (see Corollary \ref{cor:compar}).
 By the standard compactness theorems  
 for sequences of subharmonic functions (see Theorems 3.2.12, 3.2.13 and 3.4.14 in \cite{hormander}), we conclude  that 
 $u(\la, \boldsymbol{\gamma}, n)$ converges   to $\chi_{\rm Brown}(\rho_\la)$ in $L^1_{\rm loc}(\La)$. Indeed, if  $\boldsymbol{\gamma}$ is such that convergence holds a.e in $\La$, we can extract a subsequence converging  to some $v$ 
 in $L^1_{\rm loc}$, and $\limsup_{n\cv\infty}  u(\la, \boldsymbol{\gamma}, n) = v$ outside a polar set, 
 therefore $v = \chi_{\rm Brown}$. This finishes the proof.\qed

 \begin{rmk}\label{rmk:torsion equidist}
As for Theorem \ref{thm:counting}, the statement of  Theorem \ref{thm:equidist} may be adapted to the case where $\Gamma$ 
admits non-trivial torsion. 
\end{rmk}
 
 \subsection{Random closed geodesics}\label{subs:models}
Let $\mathrm{CG}(X)$ be the set of closed geodesics in $X$, which is a countable set. We allow a geodesic to be travelled several times. If this does not happen, the geodesic is said to be {\em primitive}.

Let   $\rho\in \mathrm{Hom_{par}}(G, \PSL)$. If  $\gamma\in \mathrm{CG}(X)$    
the    holonomy  $h_\rho({\gamma})$  is defined   up to conjugacy in $\rho(\Gamma)$, so it 
  makes sense to talk about its (squared) trace 
$\tr^2  h_\rho({\gamma})$. If now $(\rho_\la)$ is family of parabolic representations,
for $t\in \cc$ we may consider the subvariety in $\La$ defined by 
$$Z(\gamma, t) = \set{\la\in \La,\ \tr^2  h_{\rho_\la} ({\gamma}) = t}.$$

Here we first briefly review several notions of  ``random closed geodesics" on $X$, and interpret Theorem \ref{thm:equidist} in these terms. 

\subsubsection{Thurston model} To start with,  assume that $X$ is compact. 
The following model for random geodesics on $X$ was apparently popularized by Thurston 
(see e.g. \cite{bonahon}): for $(x,v)\in SX$ let as before $\gamma_{(x,v)}$ be the associated unit 
speed   geodesic ray. For $t>0$, close the path $\gamma_{(x,v)}\rest{[0,t]}$ with a shortest 
path returning to $x$. Notice that there may be a choice involved here, which we may   disregard   for  it happens only for a set of $(x,v)$ of zero measure. We denote by   
 $\mathrm{cg}{(x,v)}(t)$   the unique 
closed geodesic freely homotopic to the obtained loop. Notice  that its length may be much smaller than $t$. Nevertheless
it follows from Theorem \ref{thm:counting} applied to the identity representation
 that this phenomenon happens only on a set of exponentially small measure.   

In this way we have   constructed a family of projections $v\mapsto \mathrm{cg}{(x,v)}(t)$ from 
 $SX$ to  $\mathrm{CG}(X)$ 
indexed by $t$. Projecting the (normalized)  Liouville  
measure  gives rise to a family of probability measures $(m^{\rm geodesic}_t)_{t>0}$ on $\mathrm{CG}(X)$.

\medskip

Several variants of this construction may be considered:
\begin{itemize}
\itm Variant 1: a first possibility is to restrict to geodesic paths starting 
 from the marked point $\star$, and close them so as to obtain loops based at $\star$. 
\itm Variant 2:   allow for paths starting from any $(x,v)$, but append them at the two endpoints
 with shortest paths joining them to $\star$.
\end{itemize}
Notice that in these two variants, we naturally obtain   a family 
 of  measures  on the fundamental group, which in turn projects to a family of measures on $\mathrm{CG}(X)$. 
 Notice also that in the second variant, the family of  measures that we obtain in $\mathrm{CG}(X)$ is
  the same as in the original Thurston construction (since the corresponding loops are conjugate in $\Gamma$). 
  In any case,  
  the measures  $(m^{\rm geodesic}_t)_{t>0}$ may   be thought of as measures on $\Gamma$. 
A good intuitive picture for $m_t^{\rm geodesic}$ is that of being loosely equidistributed  
on an annulus of  bounded width about  the 
circle of radius $t$ in $\Gamma\subset \hh$.

 These variants are interesting because they also make sense when $X$ is only of finite volume. Indeed, using these constructions
  we obtain a family of measures   $(m^{\rm geodesic}_t)_{t>0}$ on $G$. 
  Now $\gamma\in \Gamma$ projects onto a closed geodesic if and only if it is not  parabolic. 
  Again by  Theorem \ref{thm:counting}    this happens outside a set  of exponentially small measure.
So what we do is simply to   restrict to loxodromic elements, renormalize the mass
  and project  to $\mathrm{CG}(X)$ we also 
 obtain in this way a reasonable model for random closed geodesics on $X$. 

%
%
 
  The following lemma asserts that the measures $(m^{\rm geodesic}_t)$ are well spread in $\Gamma$. 
   
  \begin{lem}\label{lem:well spread}
Let $X$ be a hyperbolic Riemann surface of finite type endowed with its hyperbolic metric, and 
  let $(m^{\rm geodesic}_t)_{t>0}$ be any of the family of measures on $\Gamma$ constructed above. 
  Then for every $\delta>0$ there exist   constants  $K$ and 
  $c$ such that for all $t>0$, for every subset $A\subset \Gamma$ with cardinality
  $\#A \leq e^{(1-\delta)t}$, we have that  $ m^{\rm geodesic}_t (A)\leq K e^{-ct}$.
  \end{lem}
  
  \begin{proof}
Let us treat the case of paths starting from $\star$  (Variant 1 above). If $v$ is a unit tangent vector at $\star$, 
we denote by $\gamma_v$ the unit speed   geodesic ray  starting from $\star$ in the direction of $v$. 

The first claim is that there exists  a constant  $K_1$  depending only on $X$ such that the  probability  
that $d_X(\gamma_v(t), \star)$ is larger than $ \frac{\delta}{2} t$ is not greater than 
$K_1e^{-   {\delta t}/{2}}$.  A proof of this fact is outlined in \cite[\S 9]{sullivan disjoint spheres}, and discussed 
with  much greater detail and generality in \cite{kleinbock margulis}.

\medskip

For $t>0$, let $E_t$ be the set of tangent vectors $v\in T^1_\star X$ such that $d_X(\gamma_v(t), \star)\leq  
\frac{\delta}{2} t$, which  by our previous claim has 
measure $\geq 1-  K_1e^{ {\delta t}/{2}}$. We lift the situation to $\hh$, so that our geodesics are now rays of hyperbolic length $t$ issued from $0$. If a lattice point  $\gamma \in \Gamma\cdot 0$  is the endpoint of $\widetilde\gamma_v(t)$ 
for some $v\in E_t$, then  $d_\hh(\gamma, \gamma_v(t))\leq \frac{\delta}{2} t$. We apply 
the first cosine rule \cite[\S 7.12]{beardon} in the   triangle 
$T(0, \gamma, \widetilde \gamma_v(t))$ to get an estimate of the angle $ \theta$ 
at the origin between the rays 
$\gamma_v(t)$ and $[0,\gamma]$. 
It follows that   $\cos( \theta)   = 1- C e^{-(2+\alpha)t}+ O(e^{-3t})$, with $\abs{\alpha}\leq \delta$, from which we deduce that 
$\theta \leq K_2 e^{-(1-\delta/2) t}$. Hence for any subset $A $ of $\Gamma$, the set of  $v\in E_t$ such that 
$\widetilde\gamma_v(t)$ belongs to  $A$ has measure at most  $K_2 (\#A) e^{-(1-\delta/2) t}$.

We conclude that if $A$ is as in the statement of the lemma, the measure of the set of tangent vectors $v$ such that 
$  \widetilde\gamma_v(t)$ belongs to  $A$ 
is not greater than $( K_1 + K_2) e^{-\frac{\delta}{2} t}$ and the result follows.
\end{proof}

It is an entertaining  exercise in hyperbolic geometry to show that the cardinality of the 
set of $\gamma\in B_\Gamma(t)$ which are proper powers of elements of $\Gamma$  
is of order of magnitude $e^{t/2}$. It follows that  a generic geodesic with 
respect to the measure $m^{\mathrm{geodesic}}_t$ is primitive. 
  
\subsubsection{Brownian Thurston model}
We can  do exactly the same construction   by using Brownian paths instead of geodesic ones. 
As before, several variants are possible, by choosing the initial point to be either fixed or 
 distributed according to the Poincar\'e 
volume element on $X$, and by choosing the paths to be closed to become a loop based at $\star$ or not. In this way we obtain   
families of measures on $\Gamma$ or on  
$\mathrm{CG}(X)$, which, abusing slightly, will be given the same notation 
 $(m_t^{\rm Brownian})_{t>0}$.
In the unit disk picture, the corresponding measures on $\Gamma$ are induced by the heat flow on $\hh^2$, so they 
 now look like Gaussian measures concentrated around   the circle of radius $t$. 

Again, these measures are well spread on $\Gamma$. 

 \begin{lem}\label{lem:well spread2}
Let $X$ be as above and  $(m^{\rm Brownian}_t)_{t>0}$ be the family of measures on $\Gamma$ 
    obtained from the Brownian model for random geodesics.  
  Then for every $\delta>0$ there exist   constants  $K$ and 
  $c$ such that for all $t>0$, for every subset $A\subset \Gamma$ with cardinality
  $\#A \leq e^{(1-\delta){t}}$, it holds that  $ m^{\rm Brownian}_t (A)\leq K e^{-ct}$.
  \end{lem}

The proof uses arguments similar to those of Theorem \ref{thm:counting}, and 
will only be sketched. Details will be left as an exercise to the reader.  

We need the following estimate. 

\begin{lem}
Let $(X, \star)$ be as above. There exists a positive constant $c$ such that 
$$\pp_\star\lrpar{d_X(\star, \omega(t))\geq ct} \lesssim e^{-ct/2}.$$
\end{lem}

\begin{proof}[Proof (sketch)]
Let $f$ be the density of the diffusion of $\delta_\star$ at time 1, relative to the hyperbolic area element on $X$. The classical Li-Yau estimates on the heat kernel (see e.g. \cite[\S B.7]{candel conlon}) imply that $f\in L^\infty(X)\subset L^2(X)$. 
Let $(\Pi_t)_{t>0}$ be the Brownian semigroup on $X$, and 
put $V_t = \set{x\in X, \ d_X(x, \star)\geq ct}$. We have to show that $\int_{V_t} \Pi_t f $ is exponentially small with $t$.  For this, using the Schwarz inequality we write 
$$\abs{\int_{V_t} \Pi_t f}\leq \norm{\Pi_t f}_{L^2(X)} \vol(V_t)^{1/2},$$
 and we conclude using the fact that $\norm{\Pi_t f}_{L^2(X)}$ converges, while $\vol(V_t)$ decreases exponentially with $t$. 
 \end{proof}
 
\begin{proof}[Proof of Lemma \ref{lem:well spread2} (sketch)]
We lift the situation to $\hh$, and look at Brownian paths issued from 0. Fix a constant $c>0$. 
By the estimates \eqref{eq:davies} for the heat kernel on $\hh$ and the above lemma, the probability that 
$\omega(t)\notin A_\hh\lrpar{t(1-c),  t(1+c)}$ or 
$d_\hh(\omega(t),\Gamma )\geq ct$ decreases exponentially with $t$. 
We discard this set of paths, so that it is enough to work with $A' = 
  A\cap A_\hh\lrpar{t(1-3c),  t(1+3c)}$ instead of $A$.
    Now if we denote by $\widetilde \omega_t$ the concatenation of $\omega\rest{[0,t]}$ 
with a shortest path joining $\omega(t)$ to $\Gamma$, we infer that 
 $\pp\lrpar{\widetilde\omega_t\in A'}\leq \nu_t
 \lrpar{\bigcup_{a\in A'} B(a, ct)},$ 
 where $\nu_t$ is the law of Brownian motion at time $t$. We compare this probability 
 to the hyperbolic area of $\bigcup_{a\in A} B(a, ct)$   by using  \eqref{eq:davies} again, and get an estimate of the form 
  $\pp\lrpar{\widetilde\omega_t\in A'} \lesssim \exp((-\delta+5c) t)$. Finally,
  choosing the constant $c$  small enough with respect to $\delta$    finishes  the proof. 
\end{proof}

 \subsubsection{Length-based model}    Let $\mathrm{CG}_{t}(X)$ be the set of primitive 
 geodesics  of length  at most $t$ on $X$, and let $m_t^{\mathrm{length}}$
 be simply  the normalized counting measure on   $\mathrm{CG}_{t}(X)$.
    
  To apply Theorem \ref{thm:equidist}  we need to 
   to lift $m_t^{\mathrm{length}}$ to $\Gamma$. For this, we   choose a shortest path from $\star$ to 
  a given geodesic to make it a loop based at $\star$. Observe that the length of this path is uniformly bounded. 
   Indeed, recall from Proposition \ref{prop:lipschitz} that there are neighborhoods $C_i$ of the cusps such that each geodesic must enter the compact region 
  $X\setminus \bigcup C_i$.

   We then infer that the lifted measure is supported 
  in an annulus of bounded width about the circle of radius $t$ in $\Gamma$.    Also, the lifted measure is well spread in $\Gamma$ in 
  the sense of Lemmas \ref{lem:well spread} and \ref{lem:well spread2} because of the well-known asymptotic estimate
 $\#\mathrm{CG}_{t}(X)\sim \frac{e^t}{t}$ as $t\cv\infty$. 
 This  estimate  originates  in the works of Huber and Selberg  
   (references covering the non-compact case include   \cite{hejhal, iwaniec})

  \subsubsection{Equidistribution} 
 We now  choose any of the above 
   models for random geodesics, and denote by $(m_t)$ the associated family of measures on $\mathrm{CG}(X)$.
Let us recall the statement of 
   the equidistribution theorem on $\La$ associated to random closed geodesics (Theorem \ref{theo:equidist}).

  \begin{thm}\label{thm:equidist geodesics}
       Let $X$ be a hyperbolic Riemann surface of finite type, and 
   $(\rho_\la)_{\la\in \La}$ be a holomorphic family of representations of $G=\pi_1(X)$,  satisfying (R1, R2, R3). 
   Let $T_{\rm bif} = dd^c\chi_{\rm Brown}$ be the natural bifurcation current. 
 
    Let $(t_n)$ be a  sequence of positive numbers such that for every $c>0$ the series $\sum e^{-ct_n}$ converges.  
   Let $(\gamma_n)$ be a random 
   sequence of closed geodesics on $X$,   according to the product measure $\prod m_{t_n}$. Then almost surely we have that 
   $$\unsur{2 \length(\gamma_n)} [   Z(\gamma_n, t)] \underset{n\cv\infty}\longrightarrow \tbif.$$ 
   \end{thm}
  
  The proof of the theorem is identical to that of Theorem \ref{thm:equidist}, using the fact that the lifted measures $m_{t_n}$ are well spread on $\Gamma$.  Indeed for fixed $\lambda$, a Borel-Cantelli argument based on 
    Theorem \ref{thm:counting} together with 
  Lemma \ref{lem:well spread} (Thurston model), Lemma \ref{lem:well spread2} (Brownian Thurston model) or the estimate 
       $\#\mathrm{CG}_{t}(X)\sim \frac{e^t}{t}$ (length-based model) implies that 
       almost surely 
 $$\unsur{2\length(\gamma_n)} \log\abs{\tr^2(\rho (\gamma_n))}\underset{n\cv\infty}{\longrightarrow}
\chi_{\rm Brown}(\rho), $$  which is the statement of 
Theorem \ref{theo:counting}.   Theorem \ref{thm:equidist geodesics} (i.e. Theorem \ref{theo:equidist}) then follows exactly 
as in \S\ref{subs:subharmonic}.

 \section{Further remarks, open problems}\label{sec:further}
 
 \subsection{} A first natural problem is to study the dependence of the natural Lyapunov exponent $\chi_{\rm Brown}$ 
 as a function of the Riemann surface structure $X$. It is likely that this function is continuous in $(\rho, X)$. It would be interesting to study the possible subharmonicity properties of this function as $X$ varies in Teichm\"uller space. 
 
 \subsection{} (cf. \cite[\S 5.2.3]{kleinbif}). The varieties  $Z(\gamma, 4)$ appearing in    the equidistribution theorems are made  
  of accidental parabolics and 
new relations. The question is to determine if  one  of the two types dominates.
At first sight one might guess that the varieties of  accidental   relations $\set{\la, \rho_\la(\gamma)=\mathrm{id}}$ 
are of codimension 3 in the representation variety. Unfortunately this is wrong. 
Indeed, if $\tr^2(\rho(\gamma)) =0$ (a codimension 1 condition)
 then $\rho(\gamma^2) = \mathrm{id}$. More generally for any $k$ 
we can obtain hypersurfaces  of representations such that $\rho(\gamma^k) = \mathrm{id}$. On the other hand, 
as observed in \S\ref{subs:models}, random geodesics are generically primitive, so this phenomenon is negligible 
(see \cite{lubotzky} for a wide generalization of this fact). 

Is it true that for a generic sequence $(\gamma_n)$ of closed geodesics, $\set{\la, \rho_\la([\gamma_n]) = \mathrm{id}}$ is asymptotically of codimension greater than 1 in the representation variety? 

\subsection{} Our results may be recast in the more general framework of holomorphic families of 
$\PSL$-cocycles over a  discrete or continuous time dynamical system, endowed with some invariant measure. 
A general question in this area of research is how the values of this cocycle over the set of periodic orbits approximates the Lyapunov exponent 
(see e.g. \cite{kalinin}).
What Theorem \ref{theo:counting} says is that we can have very precise such results (with exponential deviation estimates) in the 
particular case of the geodesic flow in constant negative curvature. It would be interesting to know how this can be generalized to a 
more general (say, uniformly hyperbolic) base dynamics. One possible way of doing this would be to use the theory of  Markov driven random products of 
matrices, which was developed in particular by Guivarc'h and his collaborators (see \cite{guivarch} for a review). 

\subsection{} What can be said about the distribution of $[Z(\gamma_n, 4)]$ when the $\gamma_n$ are {\em simple} 
closed geodesics on $X$?

\end{document}